\documentclass[11pt,reqno]{amsart}
\usepackage[utf8]{inputenc}

\usepackage{amssymb, amsmath, amsthm}
\usepackage[bookmarks, bookmarksdepth=2, colorlinks=true, linkcolor=blue, citecolor=blue, urlcolor=blue]{hyperref}

\usepackage[alphabetic,lite]{amsrefs}
\usepackage{verbatim}
\usepackage{amscd}   
\usepackage[all]{xy} 
\usepackage{youngtab} 
\usepackage{young} 
\usepackage{ytableau}
\usepackage{tikz}
\usepackage{ mathrsfs }
\usepackage{cases}
\usepackage{array}
\usepackage{cellspace}
\usepackage{calligra,mathrsfs}
\usepackage{bm}
\usepackage{graphicx}
\usepackage{rank-2-roots}
\usepackage{float}
\usepackage{enumitem}

\usepackage[prependcaption,textsize=scriptsize]{todonotes}

\newcommand{\defi}[1]{{\bf\upshape\sffamily #1}}

\DeclareMathOperator{\ShHom}{\mathscr{H}\text{\kern -3pt {\calligra\large om}}\,}

\renewcommand{\a}{\alpha}
\renewcommand{\b}{\beta}
\newcommand{\bw}{\bigwedge}

\newcommand{\bg}{{\bf g}}

\newcommand{\bG}{{\bf G}}
\newcommand{\bh}{{\bf h}}

\def\kk{{\mathbf k}}

\newcommand{\bN}{{\bf N}}

\newcommand{\onto}{\twoheadrightarrow}
\newcommand{\oo}{\otimes}
\newcommand{\pd}{\partial}

\newcommand{\Ext}{\operatorname{Ext}}

\newcommand{\GL}{\operatorname{GL}}
\newcommand{\Hom}{\operatorname{Hom}}

\newcommand{\Spec}{\operatorname{Spec}}
\newcommand{\Sym}{\operatorname{Sym}}

\newcommand{\coker}{\operatorname{coker}}
\renewcommand{\det}{\operatorname{det}}

\renewcommand{\ker}{\operatorname{ker}}

\newcommand{\bb}[1]{\mathbb{#1}}

\renewcommand{\rm}[1]{\textrm{#1}}
\newcommand{\mc}[1]{\mathcal{#1}}
\newcommand{\mf}[1]{\mathfrak{#1}}
\newcommand{\ol}[1]{\overline{#1}}
\newcommand{\op}[1]{\operatorname{#1}}

\newcommand{\ul}[1]{\underline{#1}}

\def\PP{{\mathbf P}}
\def\lra{\longrightarrow}

\newtheorem{theorem}{Theorem}[section]
\newtheorem*{theorem*}{Theorem}
\newtheorem*{problem*}{Problem}
\newtheorem{lemma}[theorem]{Lemma}

\newtheorem{proposition}[theorem]{Proposition}
\newtheorem{corollary}[theorem]{Corollary}
\newtheorem*{corollary*}{Corollary}

\newtheorem*{main-thm*}{Main Theorem}
\newtheorem*{linear-resolutions*}{Theorem on Linear Resolutions}
\newtheorem*{regularity-powers*}{Theorem on Regularity}
\newtheorem*{injectivity-Ext*}{Theorem on Injectivity of Maps of Ext Modules}
\newtheorem*{Kodaira*}{Kodaira Vanishing for Determinantal Thickenings}

\theoremstyle{definition}

\newtheorem*{definition*}{Definition}
\newtheorem{example}[theorem]{Example}

\theoremstyle{remark}
\newtheorem{remark}[theorem]{Remark}
\newtheorem*{remark*}{Remark}

\numberwithin{equation}{section}

\begin{document}

\title[Cohomology on the incidence correspondence]{Cohomology on the incidence correspondence and related questions}

\author{Annet Kyomuhangi}
\address{Department of Mathematics, Busitema University, P.O. Box 236, Tororo}
\email{annet.kyomuhangi@gmail.com}

\author{Emanuela Marangone}
\address{Dipartimento di Matematica, Università degli Studi di Genova, Via Dodecaneso, 35, 16146 Genova, Italy \newline
\indent CIMAT - Centro de Investigación en Matemáticas, Valenciana, 36023 Guanajuato, Mexico}
\email{emanuela.marangone@edu.unige.it emanuela.marangone@cimat.mx}

\author{Claudiu Raicu}
\address{Department of Mathematics, University of Notre Dame, 255 Hurley, Notre Dame, IN 46556\newline
\indent Institute of Mathematics ``Simion Stoilow'' of the Romanian Academy}
\email{craicu@nd.edu}

\author{Ethan Reed}
\address{Academy of Mathematics and Systems Science, Chinese Academy of Sciences, No. 55 Zhongguancun
East Road, Beijing, 100190, China.}
\email{ethan.reed@amss.ac.cn}


\date{\today}


\begin{abstract} 
We study a variety of questions centered around the computation of cohomology of line bundles on the incidence correspondence (the partial flag variety parametrizing pairs consisting of a point in projective space and a hyperplane containing it). Over a field of characteristic zero, this problem is resolved by the Borel--Weil--Bott theorem. In positive characteristic, we give recursive formulas for cohomology, generalizing work of Donkin and Liu in the case of the $3$-dimensional flag variety. In characteristic $2$, we provide non-recursive formulas describing the cohomology characters in terms of truncated Schur polynomials and Nim symmetric polynomials. The main technical ingredient in our work is the recursive description of the splitting type of vector bundles of principal parts on the projective line. We also discuss properties of the structure constants in the graded Han--Monsky representation ring, and explain how our cohomology calculation characterizes the Weak Lefschetz Property for Artinian monomial complete intersections.
\end{abstract}

\maketitle

\section{Introduction}\label{sec:intro}

The goal of this paper is to study character and dimension formulas for the sheaf cohomology groups of line bundles on the \defi{incidence correspondence} -- the variety parametrizing pairs consisting of a point in projective space and a hyperplane containing it. In the process we uncover surprising connections with a number of other questions that are of independent interest, such as the problems of describing:
\begin{itemize}
    \item the splitting type of vector bundles of principal parts on the projective line,
    \item the multiplication in the graded version of the Han--Monsky representation ring, or
    \item the weak Lefschetz property for Artinian monomial complete intersections.
\end{itemize}
All these problems are well-understood over a field of characteristic zero: cohomology is described by the Borel--Weil--Bott theorem, the bundles of principal parts are semi-stable (in the range relevant for us), Han--Monsky multiplication amounts to Pieri's rule for tensor products of irreducible $\mf{gl}_2$-representations, and the Lefschetz properties are known to hold by a celebrated result of Stanley. In contrast, all of these problems are open over a field of positive characteristic. The main result of our paper gives a simple recursive formula for the characters and dimensions of the cohomology of line bundles on the incidence correspondence. In characteristic~$2$, we give an alternative non-recursive formula involving Nim symmetric polynomials and truncated Schur polynomials. Prior to our work, recursive formulas were known for the $3$-dimensional flag variety \cites{donkin,liu,gao-raicu}, and partial results describing cohomology in higher dimension appear in \cites{liu-flag1,liu-polo,gao-raicu}. A complete characterization of the (non-)vanishing behavior of cohomology was previously obtained in \cite{gao-raicu}, generalizing the classical results of Griffith for the $3$-dimensional flag variety \cite{griffith}. As far as we know, our results provide the first infinite family of (partial) flag varieties of Picard rank $2$ where the cohomology is described in every characteristic! Partial results regarding (non-)vanishing of cohomology for flag varieties of type $B_2$ and $G_2$ appear in \cite{andersen-b2}*{Section~5} and \cites{humphreys-G2,and-kan}. For a recent survey regarding the general study of cohomology of line bundles on flag varieties see \cite{andersen-survey}, and for a selection of open questions, some of which will be solved in the current work, see \cite{GRV}.

Throughout the paper we let $V=\kk^n$ where $\kk$ is an algebraically closed field of characteristic $p>0$, let $S=\Sym(V) \simeq \kk[x_1,\cdots,x_n]$ denote the symmetric algebra of $V$, and write $\PP=\bb{P}V=\op{Proj}(S)$ for the corresponding projective space parametrizing $1$-dimensional quotients of $V$. Following \cite{gao-raicu} (see \eqref{eq:coh-divided=coh-incidence}), the calculation of cohomology of line bundles on the incidence correspondence is equivalent to that of cohomology of twists of divided powers of the universal subsheaf $\mc{R}$ on $\PP$. More precisely, $\mc{R}$ can be described as either $\Omega(1)$, where $\Omega$ denotes the cotangent sheaf on $\PP$, or via the tautological exact sequence
\[ 0 \lra \mc{R} \lra V \oo \mc{O}_{\PP} \lra \mc{O}_{\PP}(1) \lra 0.\]
The divided power $D^d\mc{R}$ is the subsheaf of $\mc{R}^{\oo d}$ of $\mf{S}_d$-invariant sections, where $\mf{S}_d$ is the symmetric group acting by permuting the factors on the tensor power $\mc{R}^{\oo d}$.

\medskip

\noindent{\bf Recursive character formulas.} Recall that every finite dimensional representation $W$ of the algebraic torus $(\kk^{\times})^n$ has an eigenspace decomposition (or equivalently a $\bb{Z}^n$-grading), and we define the \defi{character} of $W$ to be the Laurent polynomial
\[[W] := \sum_{(i_1,\cdots,i_n)\in\bb{Z}^n} \dim\left(W_{(i_1,\cdots,i_n)}\right)\cdot z_1^{i_1}\cdots z_n^{i_n} \in \bb{Z}[z_1^{\pm 1},\cdots,z_n^{\pm 1}].\]
If $W$ is a $\GL_n$-representation, then $[W]$ is invariant under the action of $\mf{S}_n$ by coordinate permutations, that is it belongs to the ring of symmetric Laurent polynomials (the \defi{character ring})
\begin{equation}\label{eq:def-Lambda} \Lambda = \bb{Z}[z_1^{\pm 1},\cdots,z_n^{\pm 1}]^{\mf{S}_n}.
\end{equation}
The main goal of our paper is to understand the characters
\[ h^i(D^d\mc{R}(e)) := \left[H^i(\bb{P}V,D^d\mc{R}(e))\right]\]
which are non-trivial (and highly dependent on $\op{char}(\kk)$) when $i=0,1$, and $d\geq 0$, $e\geq -1$. One has (see~\eqref{eq:coh-P-vs-Pdual})
\[ h^i(D^d\mc{R}(e)) = h^{1-i}(D^{e+1}\mc{R}(d-1))\quad\text{ for }i=0,1,\]
so we can restrict when convenient to $e\geq d-1$. To describe our results we need to introduce a number of important symmetric polynomials. The \defi{complete symmetric polynomials $h_d$} are defined as
\[ h_d = [\Sym^d V] = \sum_{\substack{i_1+\cdots+i_n = d \\ i_j\geq 0}}z_1^{i_1}\cdots z_n^{i_n}\]
and the \defi{Schur polynomials $s_{(a,b)}$} are given by
\[s_{(a,b)} = [\bb{S}_{(a,b)}V] = h_a\cdot h_b-h_{a+1}\cdot h_{b-1}.\]
If $d<p$ then the non-vanishing cohomology in degrees $i=0,1$ is given  (as in characteristic zero) by
\begin{equation}\label{eq:h01-for-small-d} 
h^0(D^d\mc{R}(e)) = s_{(e,d)}\quad\text{ for }e\geq d,\qquad h^1(D^d\mc{R}(e)) = s_{(d-1,e+1)}\quad\text{ for }e\leq d-2.
\end{equation}
We write $h'_d=h_d^{(p)}$ for the \defi{$p$-truncated complete symmetric polynomial} 
\[h'_d = \sum_{\substack{i_1+\cdots+i_n=d \\ 0\leq i_j<p}} z_1^{i_1}\cdots z_n^{i_n}\]
and define for $d\geq 0$ and $e\in\bb{Z}$ (with the convention $h'_i=0$ for $i<0$)
\[ \Phi_{d,e} = \sum_{j\geq 0} \left(h'_{e+jp}\cdot h'_{d-jp} - h'_{e+1+jp}\cdot h'_{d-1-jp}\right).\]
For $q=p^k$ we let $F^q:\Lambda\lra\Lambda$ be the endomorphism sending $z_i \mapsto z_i^q$. The following is the main result of our paper.

\begin{theorem}\label{thm:coh-recursion}
 If $e\geq d-1$ then we have for $i=0,1$ that
 \[ h^i(D^d\mc{R}(e)) = \sum_{a,b} \Phi_{d-ap,e-bp}\cdot F^p\left(h^i(D^a\mc{R}(b))\right).\]
\end{theorem}

In the sum in Theorem~\ref{thm:coh-recursion}, $b$ and $e-bp$ may be negative, but one has
\[h^i(D^a\mc{R}(b))=0\text{ if }b \leq -2,\quad\text{and}\quad \Phi_{d-ap,e-bp} = 0\text{ if }(d-ap)+(e-bp) < 0\text{ or if }d-ap<0. \]
It follows that we can always restrict the parameters $a,b$ to $0\leq a\leq d/p$, and $-1\leq b\leq (d+e)/p$.

\begin{example}\label{ex:small-rec-h1Ddr}
    Suppose that $n=5$, $p=2$. Using Theorem~\ref{thm:coh-recursion}, along with \eqref{eq:h01-for-small-d}, we get
    \[ 
    \begin{aligned}
    h^1(D^3\mc{R}(2)) &= \Phi_{1,4}\cdot F^p\left(h^1(\mc{R}(-1))\right) = \Phi_{1,4} \cdot F^p(s_{(0,0)}) = h'_4\cdot h'_1 - h'_5\cdot h'_0 \\
    &= \left(\sum z_1^2 z_2z_3z_4\right) + 4\cdot z_1z_2z_3z_4z_5,
    \end{aligned}
    \]
    where the above sum is over the $\mf{S}_5$-orbit of the monomial $z_1^2 z_2z_3z_4$. It follows that the cohomology group $H^1(\PP^4,D^3\mc{R}(2))$ is a vector space of dimension $24$. If instead we take $d=2$, $e=3$, then we get
    \[h^1(D^2\mc{R}(3)) = \Phi_{0,5}\cdot F^p\left(h^1(\mc{R}(-1))\right) = z_1z_2z_3z_4z_5,\]
    so $H^1(\PP^4,D^2\mc{R}(3))$ is $1$-dimensional. In characteristic $p=3$ we have
    \[h^1(D^3\mc{R}(2)) = \Phi_{0,5}\cdot F^p\left(h^1(\mc{R}(-1))\right) = h'_5 = \left(\sum z_1^2 z_2^2z_3\right) + \left(\sum z_1^2 z_2z_3z_4\right) +  z_1z_2z_3z_4z_5\]
    which is the character of a $51$-dimensional representation, while $h^1(D^2\mc{R}(3))=0$. For all other characteristics, we have $h^1(D^3\mc{R}(2)) = h^1(D^2\mc{R}(3)) = 0$.
\end{example}

It is easy to use the recursion in Theorem~\ref{thm:coh-recursion} to find the cohomology characters when $d$ is small. We do so for $d<p^2$ in Theorem~\ref{thm:small-weights}, which gives a proof of \cite{GRV}*{Conjecture~5.1}.

\medskip

\noindent{\bf Cohomology and Nim (characteristic $2$).} We now turn to the question of providing non-recursive formulas for cohomology, and explain the solution in characteristic $p=2$. To describe the results, we introduce the \defi{$q$-truncated complete symmetric polynomials $h_d^{(q)}$} and \defi{$q$-truncated Schur polynomials $s^{(q)}_{(a,b)}$}, defined by
\[
h^{(q)}_d = \sum_{\substack{i_1+\cdots+i_n = d \\ 0\leq i_j<q}}z_1^{i_1}\cdots z_n^{i_n},\qquad\text{and}\qquad s^{(q)}_{(a,b)} = h^{(q)}_a\cdot h^{(q)}_b-h^{(q)}_{a+1}\cdot h^{(q)}_{b-1}.
\]
We also consider the $2$-adic expansion $d=(d_k\cdots d_0)_2$ of a non-negative integer $d$
\[ d= \sum_{i=0}^k d_i\cdot 2^i,\text{ with }d_i\in\{0,1\}\text{ for all }i.\]
and define the \defi{Nim-sum} $a\oplus b$ by performing addition modulo $2$ to each of the digits in the $2$-adic expansion of non-negative integers $a,b$:
\[ a\oplus b = c\text{ if and only if }a_i + b_i \equiv c_i \text{ mod }2\text{ for all }i,\]
where $a_i,b_i,c_i$ denote the digits of $a,b,c$ in the $2$-adic expansion. We define the $n$-variate \defi{Nim symmetric polynomials} via
\begin{equation}\label{eq:def-Nim-pol}
 \mc{N}_m = \sum_{\substack{i_1+\cdots+i_n=2m \\ i_1\oplus i_2\oplus\cdots\oplus i_n = 0}} z_1^{i_1} z_2^{i_2}\cdots z_n^{i_n}.
\end{equation}
With this notation, and recalling the action by iterates of Frobenius on the character ring $\Lambda$, we have the following (see Corollary~\ref{cor:coh-Ddr-using-Nim}).

\begin{theorem}\label{thm:h1coh-char2-nonrecursive}
    If $\op{char}(\kk)=2$ and $e\geq d-1$ then
 \[ h^1(D^d\mc{R}(e)) = \sum_{(q,m,j)\in\Lambda_d} F^{2q}(\mc{N}_m)\cdot s^{(q)}_{(e-(2m-2j-1)q,d-(2m+2j+1)q)},\]
 where $\Lambda_d = \left\{(q,m,j) | q=2^r\text{ for some }r\geq 1,\ m,j\geq 0,\text{ and }(2m+2j+1)q\leq d\right\}$.
\end{theorem}

Notice that Theorem~\ref{thm:h1coh-char2-nonrecursive} corrects the statement of \cite{GRV}*{Conjecture~5.2}. Revisiting the example from \cite{GRV}*{Section~5.3}, we get for $d=6$ and $e\geq 5$ that $\Lambda_d =  \{(4,0,0), (2,0,0), (2,0,1), (2,1,0)\}$, hence
\[h^1(D^6\mc{R}(e)) = s^{(4)}_{(e+4,2)} + s^{(2)}_{(e+2,4)} + s^{(2)}_{(e+6,0)} + F^4(\mc{N}_1) \cdot s^{(2)}_{(e-2,0)}.\]
The term $s^{(2)}_{(e+6,0)}$, which is just the elementary symmetric polynomial of degree $e+6$, was missing from the original conjecture, and the first time it occurs is when the number of variables is $n\geq e+6 \geq 11$. This explains why the computational verification of the example in loc. cit. was correct up to $n=10$ variables.

For a simpler example, let $d=3$, $e=2$, so that $\Lambda_d=\{(2,0,0)\}$ and our formula becomes
\[ h^1(D^3\mc{R}(2)) = s^{(2)}_{(4,1)} = h'_4\cdot h'_1 - h'_5\cdot h'_0,\]
as noted in Example~\ref{ex:small-rec-h1Ddr} (see also \cite{GRV}*{(5.4)} and \cite{gao-raicu}*{Theorem~1.6}). While a version of Theorem~\ref{thm:h1coh-char2-nonrecursive} is not known in characteristic $p>2$, some possible generalizations of Nim polynomials that could lead to non-recursive formulas for cohomology are discussed in \cite{odorney}.

\medskip

\noindent{\bf The graded Han--Monsky representation ring.} Following \cite{han-monsky}, we consider the category of finite length graded $\kk[T]$-modules $M$, and define the tensor product $M\oo_{\kk} N$ by letting $T$ act diagonally:
\[ T \cdot (m\oo n) = Tm\oo n + m\oo Tn.\]
This induces a natural multiplication on the set of isomorphism classes, and we refer to the resulting split Grothendieck ring as the \defi{graded Han--Monsky representation ring}. We write $\delta_d$ for $\kk[T]/(T^d)$, and note that every indecomposable object in the category is isomorphic to $\delta_d(-j)$ for some $j\in\bb{Z}$, where $j$ denotes the degree of its cyclic generator. We are particularly interested in the $n$-fold product $\delta_{a_1}\cdots\delta_{a_n}$, which represents the Artinian algebra
\begin{equation}\label{eq:defA} A = \kk[T_1,\cdots,T_n]/\langle T_1^{a_1},\cdots,T_n^{a_n}\rangle
\end{equation}
viewed as a $\kk[T]$-module by letting $T=T_1+\cdots+T_n$. The following example shows that the multiplication depends in subtle ways on the characteristic $p$ of $\kk$: one has
\[ \delta_3\delta_5 = \begin{cases}
\delta_7 + \delta_4(-1) + \delta_4(-2) & \text{if }p = 2, \\
\delta_6 + \delta_6(-1) + \delta_3(-2) & \text{if }p = 3, \\
\delta_5 + \delta_5(-1) + \delta_5(-2) & \text{if }p = 5, \\
\delta_7 + \delta_5(-1) + \delta_3(-2) & \text{otherwise.} \\
\end{cases}
\]
We study the Han--Monsky multiplication in Section~\ref{sec:Han-Monsky}, where we prove in particular the following (see Proposition~\ref{prop:delcj-in-product} and Corollary~\ref{cor:prod-dela-odd-Nim}).

\begin{theorem}\label{thm:delcj-in-product}
    \begin{enumerate}
    \item If $p\nmid c$ and $\delta_c(-j)$ is a summand of $\delta_{a_1}\cdots\delta_{a_n}$ then
    \[ c+2j = a_1+\cdots+a_n - (n-1).\]
    \item If $p=2$ then $\delta_{2c+1}(-j)$ appears as a summand of $\delta_{2a_1+1}\delta_{2a_2+1}\cdots\delta_{2a_n+1}$ if and only if
    \[ c = a_1\oplus a_2\oplus\cdots\oplus a_n \quad\text{ and }\quad j = a_1+\cdots+a_n-c.\]
    \end{enumerate}
\end{theorem}

The significance of Theorem~\ref{thm:delcj-in-product} is explained in the inductive procedure that we employ to study cohomology in Section~\ref{sec:recursive-cohomology}. It turns out that if we fix a weight $\ul{a}$, the non-triviality of connecting homomorphisms in an appropriate long exact sequence in cohomology can be rephrased in terms of the presence of certain summands in $\delta_{a_1}\cdots\delta_{a_n}$ (see Theorem~\ref{thm:lefschetz-Fdr} for more details). Notice that part (2) of Theorem~\ref{thm:delcj-in-product} implies that in characteristic $p=2$, $\delta_1$ appears as a summand (with some shift) in $\delta_{2a_1+1}\delta_{2a_2+1}\cdots\delta_{2a_n+1}$ if and only if $a_1\oplus a_2\oplus\cdots\oplus a_n=0$, which means that $z_1^{a_1}\cdots z_n^{a_n}$ appears as a term in a Nim polynomial of appropriate degree. This connection leads to Theorem~\ref{thm:h1coh-char2-nonrecursive} and is discussed in Section~\ref{sec:char-2-Nim}.

\medskip

\noindent{\bf Lefschetz properties.} The study of the graded Han--Monsky representation ring (unlike its ungraded counterpart) allows one to establish a close connection to the study of Lefschetz properties for the Artinian algebras in \eqref{eq:defA} (see \cite{cook} and the references therein for more background). Recall that a linear form $\ell\in A_1$ is a Weak Lefschetz element if the multiplication maps $A_d \overset{\times\ell}{\lra} A_{d+1}$ have maximal rank for all $d$, and it is a Strong Lefschetz element if $A_d \overset{\times\ell^k}{\lra} A_{d+k}$ has maximal rank for all $d,k$. One says that $A$ has the Weak Lefschetz Property (WLP) if a generic linear form $\ell$ is a Weak Lefschetz element, and has the Strong Lefschetz Property (SLP) if a generic linear form $\ell$ is a Strong Lefschetz element. Using the standard torus action on $A$, one can show that the Lefschetz properties can be tested on the linear form $\ell = T$ \cite{MMN}*{Proposition~2.2}. It is elementary to check that if we let $s=a_1+\cdots+a_n-n$ denote the socle degree of $A$ then
\begin{itemize}
    \item $A$ has WLP if and only if every summand $\delta_r(-j)$ appearing in $\delta_{a_1}\cdots\delta_{a_n}$ satisfies
    \[ j + (r-1) \geq \frac{s}{2}\]
    \item $A$ has SLP if and only if every summand $\delta_r(-j)$ appearing in $\delta_{a_1}\cdots\delta_{a_n}$ satisfies
    \[ 2j + (r-1) = s\]
\end{itemize}
so a detailed understanding of the Han--Monsky multiplication should lead to a characterization of the Lefschetz properties for Artinian monomial complete intersections. We will explore this in future work, and note that SLP has been understood based on a different set of techniques \cite{lund-nick}*{Theorem~3.8} and \cite{nick}*{Section~3}. In this paper we focus on the connections to cohomology, which concern WLP. As explained in \eqref{eq:def-supp-d-e}, the algebra $A$ fails WLP if and only if
\[ z_1^{a_1-1}\cdots z_n^{a_n-1}\text{ appears as a term in }h^1(D^d\mc{R}(e))\text{ for some }d,e\text{ with }d+e=s,\ e\geq d-1.\]
Using this, Example~\ref{ex:small-rec-h1Ddr} (together with \eqref{eq:h01-for-small-d}) shows that
\[\kk[T_1,\cdots,T_5]/\langle T_1^2,\cdots,T_5^2\rangle\text{ satisfies WLP if and only if }\op{char}(\kk)\neq 2,3.\]
Similar reasoning shows that $\kk[T_1,T_2,T_3]/\langle T_1^3,T_2^3,T_3^2\rangle\text{ fails WLP if and only if }\op{char}(\kk)=3$.

In Theorem~\ref{thm:WLP-char-2} we provide a simple criterion (using Nim sums) to test WLP in characteristic $p=2$. While we don't have a similar characterization for all $p$, we can identify when the information of the socle degree alone guarantees WLP. In Section~\ref{subsec:WLP-from-socle} we prove the following result, which in particular recovers \cite{cook}*{Proposition~3.5(ii)} and proves \cite{cook}*{Conjecture~7.4}.

\begin{theorem}\label{thm:WLP-from-socle}
 Suppose that $n\geq 3$. All the algebras \eqref{eq:defA} of socle degree $s$ satisfy WLP if and only if one of the following holds:
 \begin{enumerate}
     \item $s\leq 2p-2$.
     \item $n=3$, and $(2t+2)q-4 \leq s\leq (2t+2)q-2$ for some $t,q$ with $1\leq t<p$, $q=p^k\geq p$.
     \item $n=4$, $p=2$, and $s=6$.
 \end{enumerate}
\end{theorem}

\medskip

\noindent{\bf Relation to vector bundles of principal parts.} Given a line bundle $\mc{L}$ on a smooth variety $X$, one can associate to it \defi{vector bundles of principal parts} (or \defi{jet bundles}) $\mc{P}^k(\mc{L})$ that describe Taylor expansions up to order $k$ for germs of sections of $\mc{L}$. When $X=\PP^1$, such bundles necessary split as direct sums of line bundles, and it is natural to study their splitting type. In characteristic zero, we have $\mc{P}^k(\mc{O}_{\PP^1}(d)) \simeq \mc{O}_{\PP^1}(d-k)^{\oplus(k+1)}$ if $d\geq k$, that is these bundles are semi-stable, but the situation in characteristic $p>0$ is more subtle \cite{maak}. For our purposes it is more natural to analyze the sheaves $\mc{F}^d_r$ introduced in \eqref{eq:def-seq-Fdr}, which are dual to the bundles of principal parts: one has
\[ \left(\mc{F}^d_r\right)^{\vee} \simeq \mc{P}^{r-1}(\mc{O}_{\PP^1}(d)).\]
The main technical result of the paper is Theorem~\ref{thm:recursive-Fdr}, where we give a recursive description of the splitting type of $\mc{F}^d_r$, keeping track of the equivariant structure with respect to the natural action of the $2$-dimensional torus. This leads to recursive formulas for the cohomology of $\mc{F}^d_r$ in Section~\ref{sec:coh-Fdr-recursive}, and the fundamental observation comes in Theorem~\ref{thm:pparts-vs-Monsky} when we explain how knowledge of the (torus equivariant structure of the) said cohomology completely determines the multiplication in the Han--Monsky ring.  

\medskip

\noindent{\bf Organization.} In Section~\ref{sec:prelim} we collect a number of elementary observations relating the cohomology of line bundles on the incidence correspondence, the cohomology of $D^d\mc{R}(e)$, and the Artinian algebras \eqref{eq:defA}. In Section~\ref{sec:pparts-P1} we provide the recursive description for the splitting type of the (duals of the) bundles of principal parts on $\PP^1$, and use this in Section~\ref{sec:coh-Fdr-recursive} to extract information about their cohomology. Section~\ref{sec:Han-Monsky} is concerned with the graded Han--Monsky ring, and the connection to the cohomology calculations on $\PP^1$. In Section~\ref{sec:recursive-cohomology} we consider the natural generalizations of the sheaves $\mc{F}^d_r$ to higher dimensional projective spaces, and using their filtrations by sheaves of the form $D^d\mc{R}(e)$ we establish Theorem~\ref{thm:coh-recursion}. Section~\ref{sec:char-2-Nim} focuses on the characteristic $2$ situation, proving Theorem~\ref{thm:h1coh-char2-nonrecursive}. Finally, we end with some applications to Lefschetz properties for Artinian monomial complete intersections in Section~\ref{sec:Lefschetz}.

\section{Preliminaries}\label{sec:prelim}

Let $V=\kk^n$, where $\kk$ is an algebraically closed field of characteritic $p>0$, and consider the projective space $\PP=\bb{P}(V)=\op{Proj}(\Sym V)$ with tautological short exact sequence
\begin{equation}\label{eq:ses-on-PV}
0 \lra \mc{R} \lra V \oo \mc{O}_{\PP} \lra \mc{O}_{\PP}(1) \lra 0,
\end{equation}
where $\mc{R}=\Omega(1)$ is the tautological rank $(n-1)$ subsheaf, $\Omega$ is the cotangent sheaf, and $\mc{O}_{\PP}(1)$ is the tautological quotient line bundle. For $d\geq 0$, $e\geq -1$, \eqref{eq:ses-on-PV} gives rise to a short exact sequence
\[0 \lra D^d\mc{R}(e) \lra D^d V\oo \mc{O}_{\PP}(e) \lra D^{d-1} V\oo \mc{O}_{\PP}(e+1) \lra 0.\]
By taking global sections, we get a map
\begin{equation}\label{eq:def-Delta-de}
D^d V \oo \Sym^e V \overset{\Delta_{d,e}}{\lra} D^{d-1}V \oo \Sym^{e+1}V
\end{equation}
induced by the co-multiplication $D^dV \to D^{d-1}V\oo V$ followed by the multiplication $V\oo\Sym^e V \to \Sym^{e+1}V$. The long exact sequence in cohomology yields
\begin{equation}\label{eq:coh-Ddr-as-ker-coker}
H^i(\PP,D^d\mc{R}(e))=\begin{cases}
    \ker(\Delta_{d,e}) & \text{if }i=0, \\
    \coker(\Delta_{d,e}) & \text{if }i=1, \\
    0 & \text{otherwise,}
\end{cases}
\end{equation}
and our main goal is to understand these cohomology groups. When $e\leq -2$ the cohomology is concentrated in $H^{n-1}$ or is identically zero, and it is computed from the short exact sequence
\[ 0 \lra H^{n-1}(\PP,D^d\mc{R}(e)) \lra D^d V\oo H^{n-1}(\PP,\mc{O}_{\PP}(e)) \lra D^{d-1} V\oo H^{n-1}(\PP,\mc{O}_{\PP}(e+1))\lra 0\]
Dualizing \eqref{eq:def-Delta-de} and using the fact that $(\Sym^m V)^{\vee} = D^m(V^{\vee})$, it follows that if we consider the dual projective space $\PP^{\vee} = \bb{P}(V^{\vee})$, then we have
\begin{equation}\label{eq:coh-P-vs-Pdual}
 H^i(\PP,D^d\mc{R}(e)) = H^{1-i}(\PP^{\vee},D^{e+1}\mc{R}(d-1))^{\vee}\quad\text{ for }i=0,1.
\end{equation}
This symmetry is even more apparent if we consider the \defi{incidence correspondence}
\[ X = \{(p,H) : p\in H\} \subset \PP \times \PP^{\vee},\]
which is a hypersurface cut out by the bilinear form
\begin{equation}\label{eq:def-omega}
 \omega = x_1y_1+\cdots+x_ny_n.
\end{equation}
The line bundles on $X$ arise by restriction from $\PP \times \PP^{\vee}$, and we write
\[ \mc{O}_X(a,b) = \mc{O}_{\PP \times \PP^{\vee}}(a,b)_{|_X}\quad\text{ for }(a,b)\in\bb{Z}^2.\]
The cohomology groups \eqref{eq:coh-Ddr-as-ker-coker} can then be identified (up to a twist by the determinant representation) with cohomology groups of line bundles on $X$ (see \cite{gao-raicu}*{(2.12)}, \cite{GRV}*{Section~5}):
\begin{equation}\label{eq:coh-divided=coh-incidence} H^i\left(\PP, \op{D}^d\mc{R}(e)\right) \simeq H^{i+n-2}\left(X,\mc{O}_X(e+1,-d-n+1)\right) \oo \bw^n V^{\vee}.
\end{equation}
Under this identification, the symmetry \eqref{eq:coh-P-vs-Pdual} becomes a manifestation of Serre duality on $X$ and of the symmetric role played by $\PP$ and $\PP^{\vee}$.

It will be useful to give a more concrete realization of the cohomology groups above, and to study them for all $d,e$ simultaneously. To that end we fix a basis $x_1,\cdots,x_n$ of $V$, and the dual basis $y_1,\cdots,y_n$ of $V^{\vee}$. We consider the polynomial rings
\[ S = \kk[x_1,\cdots,x_n]\quad\text{and}\quad R = S[y_1,\cdots,y_n]\simeq\kk[y_1,\cdots,y_n,x_1,\cdots,x_n],\]
which have a coordinate independent description as
\[ S = \Sym_{\kk}V = \bigoplus_{e\geq 0}\Sym^e V,\quad R = \Sym_{\kk}(V^{\vee}\oplus V)=\bigoplus_{d,e\geq 0} \Sym^d(V^{\vee}) \oo \Sym^e V.\]
We also consider the local cohomology module
\[ M = \bw^n V^{\vee} \oo H^n_{(y_1,\cdots,y_n)}(R),\]
and define a bigrading of $M$ via
\begin{equation}\label{eq:explicit-basis-Mde} M_{d,e} = D^d V \oo \Sym^e V = \bw^n V^{\vee} \oo \bigoplus_{\substack{e_i,d_j\geq 0 \\ d_1+\cdots+d_n=d,\\ e_1+\cdots+e_n=e}} \kk\cdot\frac{x_1^{e_1}\cdots x_n^{e_n}}{y_1^{1+d_1}\cdots y_n^{1+d_n}}.
\end{equation}
Here the one-dimensional factor $\bw^n V^{\vee}$ is included in order to keep track of the equivariant structure, and can be ignored in concrete calculations. If we write $\omega = x_1y_1+\cdots+x_ny_n$ then we can identify \eqref{eq:def-Delta-de} with the multiplication map
\begin{equation}\label{eq:Delta-de-is-mult-omega} M_{d,e} \overset{\cdot\omega}{\lra} M_{d-1,e+1}.
\end{equation}
The short exact sequence
\[ 0 \lra R \overset{\cdot\omega}{\lra} R \lra R/\omega \lra 0\]
induces a long-exact sequence in local cohomology, whose only non-zero terms are
\[0 \lra H^{n-1}_{(y_1,\cdots,y_n)}(R/\omega) \lra M \overset{\cdot\omega}{\lra} M \lra H^n_{(y_1,\cdots,y_n)}(R/\omega) \lra 0,\]
hence the cohomology calculation in \eqref{eq:coh-Ddr-as-ker-coker} can be further rephrased into understanding the bi-graded components of the local cohomology groups of $R/\omega$ with support in $(y_1,\cdots,y_n)$. Our goal will then be to understand multiplication by $\omega$ on $M$, and in fact in the case $n=2$ it will also be important to analyze multiplication by powers $\omega^r$.

To that end, we will employ the symmetry coming from the natural $(\GL(V)=)\GL_n$-action on $S$ via linear change of coordinates. We view $y_1,\cdots,y_n$ as dual variables to $x_1,\cdots,x_n$ which induces a $\GL_n$-action on $R$ making the quadratic form $\omega$ a $\GL_n$-invariant. In particular, the maps \eqref{eq:Delta-de-is-mult-omega} and their iterates are $\GL_n$-equivariant. Restricting to the maximal torus of diagonal matrices in $\GL_n$, we get a $\bb{Z}^n$-grading on $R,S,M$, where $\deg(x_i)=-\deg(y_i)=\vec{e}_i$ is the $i$-th standard unit vector. Moreover, since $\bw^n V^{\vee}$ is a $1$-dimensional space concentrated in multidegree $(-1,\cdots,-1)$, we get for $\ul{a}=(a_1,\cdots,a_n)\in\bb{Z}^n$ that
\[ M_{\ul{a}} \simeq \bigoplus_{d_i+e_i=a_i}\kk\cdot\frac{x_1^{e_1}\cdots x_n^{e_n}}{y_1^{1+d_1}\cdots y_n^{1+d_n}}. \]

Finally, the following relationship to monomial Artinian complete intersections will be important in our analysis. We write $T_i = x_iy_i$ and consider the polynomial subalgebra $\kk[T_1,\cdots,T_n]$ of $R$. It therefore acts on $M$, and since $\deg(T_i)=\vec{0}$, it preserves the $\bb{Z}^n$-graded components. Moreover, each $M_{\ul{a}}$ is a cyclic $\kk[T_1,\cdots,T_n]$-module, given by
\[M_{\ul{a}} = \kk[T_1,\cdots,T_n]\cdot\frac{1}{y_1^{1+a_1}\cdots y_n^{1+a_n}} \simeq \kk[T_1,\cdots,T_n]/\langle T_1^{1+a_1},\cdots, T_n^{1+a_n}\rangle. \]
Multiplication by $\omega$ on $M_{\ul{a}}$ translates via this identification into multiplication by $T=T_1+\cdots+T_n$. If we view $M_{\ul{a}}$ as a standard graded Artinian algebra (where the variables $T_i$ have degree $\deg(T_i)=1$) then 
\begin{equation}\label{eq:multigraded-comps-coh}
    H^0(\PP,D^dR(e))_{\ul{a}} = (0 : T)_e\quad\text{and}\quad H^1(\PP,D^dR(e))_{\ul{a}} = \left(\frac{M_{\ul{a}}}{T\cdot M_{\ul{a}}}\right)_{e+1}.
\end{equation}
This provides a close relationship between cohomology, the Han--Monsky representation ring discussed in Section~\ref{sec:Han-Monsky}, and Lefschetz properties for monomial complete intersections considered in Section~\ref{sec:Lefschetz}. The connection between local cohomology and the work of Han and Monsky is also investigated in ongoing work of Hochster and Kenkel, while the study of the local cohomology groups of $R/\omega$ (over $\bb{Z}$) famously arises in the construction of local cohomology groups with infinitely many associated primes \cite{singh}*{Section~4}.

\section{Splitting of duals of principal parts}\label{sec:pparts-P1}

Let $U$ be a $\kk$-vector space of dimension $\dim(U)=2$. Consider the projective line $\PP^1=\bb{P}U$ with tautological short exact sequence
\begin{equation}\label{eq:ses-on-P1}
0 \lra \mc{R} \lra U \oo \mc{O}_{\PP^1} \lra \mc{O}_{\PP^1}(1) \lra 0,
\end{equation}
where $\mc{R}=\Omega^1_{\PP^1}(1) = \bw^2 U\oo \mc{O}_{\PP^1}(-1)$ is the universal subsheaf, and $\mc{O}_{\PP^1}(1)$ is the universal quotient sheaf. The natural action of $\GL(U)$ on $U$ induces an action on $\PP^1$ and a canonical linearization of $\mc{O}_{\PP^1}(1)$ such that \eqref{eq:ses-on-P1} is an exact sequence of $\GL(U)$-equivariant sheaves.

For each $d\geq 0$, the sequence \eqref{eq:ses-on-P1} induces a short exact sequence
\[ 0 \lra \mc{R}^{\oo d} \lra D^dU \oo \mc{O}_{\PP^1} \overset{\Delta_d}{\lra} D^{d-1}U\oo\mc{O}_{\PP^1}(1) \lra 0,
\]
and we define a family of $\GL(U)$-equivariant locally free sheaves $\mc{F}^d_r$ via
\begin{equation}\label{eq:def-seq-Fdr} 0 \lra \mc{F}^d_r \lra D^dU \oo \mc{O}_{\PP^1} \overset{\Delta}{\lra} D^{d-r}U \oo \mc{O}_{\PP^1}(r) \lra 0,
\end{equation}
where $\Delta = (\Delta_{d-r+1}(r-1)) \circ \cdots \circ (\Delta_{d-1}(1)) \circ \Delta_d$. We have that $\Delta=0$ if $d<r$, hence
\begin{equation}\label{eq:Fdr-for-d<r} 
 \mc{F}^d_r = D^dU \oo \mc{O}_{\PP^1}\quad\text{when }d<r,
\end{equation}
which is why we will usually assume that $d\geq r$. In this case we have by construction that $\mc{F}^d_r$ has a filtration with composition factors
\[ \mc{R}^{\oo d},\ \mc{R}^{\oo(d-1)}(1),\cdots,\ \mc{R}^{\oo(d-r+1)}(r-1)\]
and therefore
\begin{equation}\label{eq:rank-det-Fdr}
\op{rank}(\mc{F}^d_r) = r \quad\text{ and }\quad \det(\mc{F}^d_r) = \bw^r\mc{F}^d_r = \left(\bw^2 U\right)^{\oo\left(dr-{r\choose 2}\right)} \oo \mc{O}_{\PP^1}(-r(d-r+1)).
\end{equation}
Moreover, by snake's lemma we have natural short exact sequences
\begin{equation}\label{eq:ses-Fdrs}
    0 \lra \mc{F}^d_r \lra \mc{F}^d_s \lra \mc{F}^{d-r}_{s-r}(r) \lra 0\qquad\text{ for all }s\geq r.
\end{equation}

By Grothendieck's theorem, $\mc{F}^d_r$ splits as a direct sum of line bundles, and our goal is to understand the splitting type. We would like to take advantage of the $\GL(U)$-equivariant structure, but the splitting cannot be usually taken to be $\GL(U)$-equivariant. For the rest of the section we fix a basis of $U$ and make an identification $U=\kk^2$. We consider the $2$-dimensional torus $T=\bb{G}_m \times \bb{G}_m \subset\GL(U)$ acting coordinate-wise on~$U$. By \cite{kumar}, there exists a decomposition
\[\mc{F}^d_r = \mc{L}_1 \oplus \cdots \oplus \mc{L}_r,\]
where each $\mc{L}_i$ is a $T$-equivariant line bundle (note that although the statement in \cite{kumar} is phrased over $\bb{C}$, the proof only uses the fact that irreducible $T$-representations are $1$-dimensional, and taking $T$-invariants is an exact functor). 

\begin{remark}\label{rem:principal-parts}
    When $d\geq r-1$ one can identify $\left(\mc{F}^d_r\right)^{\vee}$ with the vector bundle of principal parts $\mc{P}^{r-1}(\mc{O}_{\PP^1}(d))$, which was analyzed in \cite{maak}. In characteristic zero, \cite{maak}*{Proposition~6.3} implies that for $d\geq r-1$
\[ \mc{F}^d_r \simeq \mc{O}_{\PP^1}(-(d-r+1))^{\oplus r}.\]
In fact, it is not hard to check that its $\GL(U)$-equivariant structure can be described via
\begin{equation}\label{eq:char-0-GLU-Fdr}
\mc{F}^d_r = (\det U)^d\oo(\Sym^{r-1}U)^{\vee}\oo\mc{O}_{\PP^1}(-(d-r+1)).
\end{equation}
As we will see, the situation in characteristic $p>0$ is much more subtle.
\end{remark}

Each $T$-equivariant line bundle on $\PP^1$ is of the form
\[L_{u,v}\oo\mc{O}_{\PP^1}(i),\text{ where }u,v,i\in\bb{Z}
\]
and $L_{u,v}=\kk\cdot\chi^{u,v}$ is the $1$-dimensional $T$-representation where the action is given by
\begin{equation}\label{eq:act-on-chiuv}
 (t_1,t_2) * \chi^{u,v} = t_1^{u}t_2^{v}\chi^{u,v}\text{ for }t_1,t_2\in\bb{G}_m.
\end{equation}
For $a\geq b\geq 0$ we consider the $T$-representation
\[ H_{a,b} = \bigoplus_{i=b}^a L_{a+b-i,i},\]
and observe that $H_{a,b}$ is isomorphic to the Schur functor $\bb{S}_{(a,b)}U$ viewed as a $T$-representation. We make the convention $H_{a,b}=0$ if $a<b$. For $p=\op{char}(\kk)$ and $q=p^e$, we consider the Frobenius action on $T$-representations, and more generally on $T$-equivariant sheaves on $\PP^1$, whose action on line bundles is given~by
\[ F^q\left(L_{u,v}\oo\mc{O}_{\PP^1}(i)\right) = L_{qu,qv}\oo\mc{O}_{\PP^1}(qi).\]
For a $T$-representation $W$ and $i\in\bb{Z}$, we will often write $W(i)$ instead of $W\oo\mc{O}_{\PP^1}(i)$. It follows from \eqref{eq:char-0-GLU-Fdr} that in characteristic $0$ we have for $d\geq r-1$
\begin{equation}\label{eq:char-0-Tequiv-Fdr}
 \mc{F}^d_r = H_{d,d-r+1}(-(d-r+1)).
\end{equation}
The main result of this section describes recursively the $T$-equivariant splitting of $\mc{F}^d_r$ in characteristic $p>0$ as follows.

\begin{theorem}\label{thm:recursive-Fdr}
    Suppose that $\op{char}(\kk)=p>0$, $q'\leq r\leq q$ where $q=p^e$, $q'=p^{e-1}$. We have the following $T$-equivariant isomorphisms:
    \begin{enumerate}
        \item If $d<r$ then 
        \[ \mc{F}^d_r = H_{d,0}\oo\mc{O}_{\PP^1}.\]
        \item If $d\geq q+r-1$ then
        \[ \mc{F}^d_r = L_{q,q} \oo \mc{F}^{d-q}_r(-q).\]
        \item If $q-1\leq d \leq q+r-1$ then
        \[ \mc{F}^d_r = H_{d,q}(-q) \bigoplus H_{q-1,d-r+1}(r-q).\]
        \item If $r\leq d\leq q-1$, and if we let $1\leq a<p$ such that $aq'\leq r<(a+1)q'$, then
        \[\mc{F}^d_r = F^{q'}(H_{a,0}) \oo \mc{F}^{d-aq'}_{r-aq'} \bigoplus F^{q'}(H_{a-1,0}) \oo 
        \mc{F}^{d-r+q'}_{(a+1)q'-r}(r-aq').\]
    \end{enumerate}
\end{theorem}

The proof of Theorem~\ref{thm:recursive-Fdr} will occupy the rest of the section. Cases (1)--(3) are somewhat straightforward, but case (4) seems to require a more delicate analysis, and the crucial ingredient is to understand the action of differential operators on the sheaves $\mc{F}^d_r$, and their associated graded modules $F^d_r$. Before going into the details of the proof, we point out the following reformulation of part (4).

\begin{corollary}\label{cor:rec-Fdr-sometimes}
 Suppose that $q'\leq r\leq q$ are as in Theorem~\ref{thm:recursive-Fdr}.
 \begin{enumerate}
  \item[(i)] If $aq'\leq r\leq d < (a+1)q'$ then
        \[\mc{F}^d_r = F^{q'}(H_{a,0}) \oo \mc{F}^{d-aq'}_{r-aq'} \bigoplus F^{q'}(H_{a-1,0}) \oo \left[H_{d-r+q',q'}(r-(a+1)q') \bigoplus H_{q'-1,d-aq'+1}\oo\mc{O}_{\PP^1}\right] \] 
  \item[(ii)] In particular, if $p=2$ and $q'\leq r\leq d<q$ then 
          \[\mc{F}^d_r = F^{q'}U \oo \mc{F}^{d-q'}_{r-q} \bigoplus H_{d-r+q',q'}(r-q) \bigoplus H_{q'-1,d-q'+1}\oo\mc{O}_{\PP^1} \] 
 \end{enumerate}
\end{corollary}

\begin{proof}
 Part (ii) is the special case of (i) when $p=2$ and $a=1$, so it suffices to consider case (i). Using Theorem~\ref{thm:recursive-Fdr}(4), the conclusion reduces to showing that
 \[ \mc{F}^{d-r+q'}_{(a+1)q'-r} = H_{d-r+q',q'}(-q') \bigoplus H_{q'-1,d-aq'+1}(aq'-r).\]
 If we set $d'=d-r+q'$, $r'=(a+1)q'-r$, then the hypothesis $aq'\leq r\leq d\leq (a+1)q'-1$ implies $q'-1\leq d'\leq q'+r'-1$, so we can apply Theorem~\ref{thm:recursive-Fdr}(3) to get the desired description of $\mc{F}^{d-r+q'}_{(a+1)q'-r} =\mc{F}^{d'}_{r'}$.
\end{proof}

We next introduce some notation to be used in the rest of the section. We let $S=\kk[x_1,x_2] = \Sym_{\kk}U$, write $S(r)$ for the rank one graded module generated in degree $-r$, and define the graded $S$-module $F^d_r$ via
\[ F^d_r = \ker\left(D^d U \oo S \lra D^{d-r}U\oo S(r)\right),\]
induced by the comultiplication $D^d U \lra D^{d-r}U \oo D^r U$, followed by the natural map $D^r U \lra \Sym^r U$ and the multiplication $\Sym^r U \oo S \lra S(r)$. Notice that $\mc{F}^d_r$ is the sheaf associated to the graded module $F^d_r$. Using the notation in Section~\ref{sec:prelim} for $n=2$ and $V=U$, we get that
\[ D^d U \oo S = M_{d,\bullet} = \bigoplus_{e\geq 0}M_{d,e}.\]
Moreover, for each $e$ we get an exact sequence
\[ 0 \lra H^0\left(\PP^1,\mc{F}^d_r(e)\right) \lra M_{d,e} \overset{\cdot\omega^r}{\lra} M_{d-r,e+r} \]
Using the explicit description \eqref{eq:explicit-basis-Mde}, we get that multiplication by $y_i$ induces an $S$-linear map $M_{d,\bullet}\lra M_{d-1,\bullet}$ which commutes with multiplication by $\omega$, hence we get induced maps
\begin{equation}\label{eq:mult-yi-Fdr}
 y_i : F^d_r \lra F^{d-1}_r\quad\text{ and }\quad y_i : \mc{F}^d_r \lra \mc{F}^{d-1}_r.    
\end{equation}

\subsection{Local charts for the sheaves $\mc{F}^d_r$.} It will be useful to have a local description of the sheaves $\mc{F}^d_r$, and because of symmetry we will only focus on the basic open affine $D_+(x_1)$, leaving the analysis in the chart $D_+(x_2)$ to the interested reader (see also \cite{maak}). We write $t=x_2/x_1$, $A=\kk[t]$ so that $D_+(x_1)=\Spec(A)$, and consider the identifications
\begin{equation}\label{eq:A-basis-Ddu}D^dU\oo\mc{O}_{\PP^1}(D_+(x_1)) = \op{Span}_A\left\{\frac{1}{y_1^{1+d_1}y_2^{1+d_2}}:d_1+d_2=d\right\},
\end{equation}
\[
D^{d-r}U\oo\mc{O}_{\PP^1}(r)(D_+(x_1)) = \op{Span}_A\left\{\frac{x_1^r}{y_1^{1+d_1}y_2^{1+d_2}}:d_1+d_2=d-r\right\}.\]
Throughout this section we will drop the factor $\bw^2 U^{\vee}$ in \eqref{eq:explicit-basis-Mde} from the notation, so the natural $T$-action on $M$ will give a $T$-equivariant identification
\begin{equation}\label{eq:forget-bw2Udual}
 \kk\cdot \frac{1}{y_1^{1+d_1}y_2^{1+d_2}} = L_{d_1,d_2}.
\end{equation}

\begin{lemma}\label{lem:local-basis-Fdr}
    If $d\geq r$ then the following sections give an $A$-basis of $\mc{F}^d_r(D_+(x_1))$:
    \[ f_i = \sum_{j=i}^{d} {-r \choose j-i}\cdot t^{j-i} \cdot\frac{1}{y_1^{j+1}y_2^{d+1-j}},\quad\text{ for }i=0,\cdots,r-1.\]
\end{lemma}

\begin{proof} Relative to the basis of $D^dU\oo A$ from \eqref{eq:A-basis-Ddu}, the $f_i$'s are described by the columns of the matrix
\[
\begin{bmatrix}
   1 & 0 & 0 & \cdots & 0 \\
   {-r\choose 1}\cdot t & 1 & 0 & \cdots & 0 \\
   {-r\choose 2}\cdot t^2 & {-r\choose 1}\cdot t & 1 & \cdots & 0 \\
   \vdots & \vdots & \vdots & \ddots & \vdots \\
   \vdots & \vdots & \vdots & \ddots & 1 \\
   \vdots & \vdots & \vdots & \ddots & \vdots \\
  {-r\choose d}\cdot t^d & {-r\choose d-1}\cdot t^{d-1} & {-r\choose d-2}\cdot t^{d-2} & \cdots & {-r\choose d-r+1}\cdot t^{d-r+1} \\
\end{bmatrix}
\]
Since the submatrix formed by the first $r$ rows has determinant $1$ determines a direct summand of $D^dU\oo A$ of rank $r$. To show that this agrees with $\mc{F}^d_r(D_+(x_1))$, it then suffices to check that each $f_i$ is annihilated by $\omega^r = (y_1+ty_2)^r\cdot x_1^r$.

To that end, we consider the identity
\[ 1 = (1+z)^{-r}\cdot (1+z)^r = \left(\sum_{u\geq 0}{-r\choose u}z^u\right) \cdot \left(\sum_{v\geq 0}{r\choose v}z^v\right),\]
which implies that for $w>0$ we have
\begin{equation}\label{eq:identity-bin*bin} \sum_{u+v=w}{-r\choose u}\cdot {r\choose v} = 0.
\end{equation}
It follows that
\[ f_i\cdot(y_1+ty_2)^r = \sum_{j=i}^d\sum_{k=0}^r {-r\choose j-i}\cdot {r\choose r-k} t^{r-k+j-i}\cdot\frac{1}{y_1^{1+(j-k)}y_2^{1+(d-r-(j-k))}}.\]
Noting that the terms where $j<k$ vanish, it follows that for the remaining terms we have $w=r-k+j-i\geq r-i>0$. Setting $u=j-i$, $v=r-k$, rearranging the terms and using \eqref{eq:identity-bin*bin} gives the desired vanishing.
\end{proof}

\subsection{The proof of cases (1)--(3) of Theorem~\ref{thm:recursive-Fdr}.} We can now apply our preliminary results to establish the recursive formulas of Theorem~\ref{thm:recursive-Fdr} in all but the last case.

{\bf Case (1)} follows by combining \eqref{eq:Fdr-for-d<r} with the $T$-isomorphism $D^dU \simeq H_{d,0}$. 

{\bf Case (2): $d\geq q+r-1$.} We will construct a $T$-equivariant complex
\begin{equation}\label{eq:cx-Fdr-into-FqFd-qr} 
0 \lra \mc{F}^d_r \overset{\iota}{\lra} F^qU \oo \mc{F}^{d-q}_r \overset{\pi}{\lra} \mc{F}^{d-q}_r(q) \lra 0
\end{equation}
where $\iota$ is injective, and $\pi$ is induced by the natural surjection $F^q U \oo \mc{O}_{\PP^1}\onto\mc{O}_{\PP^1}(q)$ and hence it is surjective. Since
\[ \op{rank}\left(F^qU \oo \mc{F}^{d-q}_r\right) = 2r = \op{rank}\left(\mc{F}^d_r\right) + \op{rank}\left(\mc{F}^{d-q}_r(q)\right),\]
the middle homology of \eqref{eq:cx-Fdr-into-FqFd-qr} must be $0$-dimensional. It follows from \eqref{eq:rank-det-Fdr} that the determinant of the complex is trivial, hence the complex is in fact exact. Conclusion (2) is then a consequence of the identification
\[ \ker\left(F^q U \oo \mc{O}_{\PP^1}\onto\mc{O}_{\PP^1}(q)\right) = \mc{R}^{\oo q} = L_{q,q}(-q).\]

To construct \eqref{eq:cx-Fdr-into-FqFd-qr} we use the basis $\{x_1^q,x_2^q\}$ of $F^qU$, as well as the $q$-iterates of the maps in \eqref{eq:mult-yi-Fdr}, in order to define $\iota$ via
\[ \iota(f) = x_1^q \oo y_1^q(f) + x_2^q \oo y_2^q(f)\quad\text{ for }f\text{ a local section of }\mc{F}^d_r.\]
Using the identity
\[ x_1^qy_1^q+x_2^qy_2^q = \omega^q,\]
together with the hypothesis $r\leq q$, it follows that
\[ \pi(\iota(f)) = \omega^q\cdot f = \omega^{q-r}\cdot \omega^r \cdot f = 0,\]
hence $\iota$ and $\pi$ define indeed a complex. To prove that $\iota$ is injective, we restrict to the affine chart $D_+(x_1)$ and use Lemma~\ref{lem:local-basis-Fdr}: it is enough to check the $A$-linear independence of the sections
\[ y_2^q(f_i) = \sum_{j=i}^{d-q} {-r \choose j-i}\cdot t^{j-i} \cdot\frac{1}{y_1^{j+1}y_2^{d+1-j-q}} = \frac{1}{y_1^{i+1}y_2^{d+1-i-q}} + \text{ lower order terms}.\]
The linear independence then follows from the fact that the displayed leading terms are distinct and non-zero, which uses $d-q\geq r-1\geq i$ for $i=0,\cdots,r-1$.

{\bf Case (3): $q-1\leq d\leq q+r-1$.} If $d\leq r-1$, then since $q-1\leq d$ and $r\leq q$, we must have $r=q$ and $d=q-1$, which by our convention makes $H_{d,q}=0$. The conclusion of (3) then becomes $\mc{F}^d_r = H_{q-1,0}\oo\mc{O}_{\PP^1}$ which is a special case of (1). We may therefore assume that $d\geq r$.

By an easy extension of the argument in \cite{rai-vdb}*{Section~4.5}, any short exact sequence of locally free sheaves
\[ 0 \lra \mc{F} \lra \mc{E} \lra \mc{L} \lra 0,\]
where $\mc{L}$ is a line bundle, gives rise for $m\geq 0$ to a right resolution $\mc{G}^{\bullet}$ of $T_qD^m\mc{F}$ given by
\[\mc{G}^{2i} = \mc{L}^{qi}\oo T_qD^{m-qi}\mc{E},\qquad \mc{G}^{2i+1} = \mc{L}^{q(i+1)-r}\oo T_qD^{m-q(i+1)+r}\mc{E}.\]
We apply this to the tautological sequence \eqref{eq:ses-on-P1}, so that $\mc{F}=\mc{R}$ is also a line bundle, $\mc{E}=U\oo\mc{O}_{\PP^1}$, and $\mc{L}=\mc{O}_{\PP^1}(1)$. It follows that if we take $m=d+q-r$, then $m\geq q$ and therefore $T_qD^m\mc{F}=0$. The complex $\mc{G}^{\bullet}$ reduces then to the exact sequence
\[0 \lra T_qD^{d+q-r}U \oo \mc{O}_{\PP^1} \lra T_qD^dU(q-r) \lra T_qD^{d-r}U(q) \lra T_q D^{d-q}U(2q-r) \lra 0.\]
Using the exactness of $\mc{G}^{\bullet}$ and \eqref{eq:def-seq-Fdr}, we get that
\[ \mc{F}^{d-r}_{q-r}(q) = \ker\left(\mc{G}^2\lra\mc{G}^3\right) = \coker\left(\mc{G}^0\lra\mc{G}^1\right).\]
Twisting by $\mc{O}_{\PP^1}(r-q)$, we can then form the following diagram with exact rows
\[
\xymatrix{
  & 0 \ar[d] & 0 \ar[d] & 0 \ar[d] & \\
0 \ar[r] & T_qD^{d+q-r}U(r-q) \ar[r] \ar[d] & T_qD^{d}U \oo \mc{O}_{\PP^1} \ar[r] \ar[d] & \mc{F}^{d-r}_{q-r}(r) \ar[r] \ar[d] & 0\\
0 \ar[r] & \mc{F}^d_r \ar[r] \ar[d] & D^{d}U \oo \mc{O}_{\PP^1} \ar[r] \ar[d] & D^{d-r}U(r) \ar[r] \ar[d] & 0\\
0 \ar[r] & D^{d-q}U \oo \mc{R}^{\oo q} \ar[r] \ar[d] & D^{d-q}U \oo F^qU\oo  \mc{O}_{\PP^1} \ar[r] \ar[d] & D^{d-q}U(q) \ar[r] \ar[d] & 0\\
  & 0  & 0  & 0  & \\
}
\]
The middle column is exact since $d<2q$ (see \cite{gao-raicu}*{(2.7)}), while the third column is exact by \eqref{eq:def-seq-Fdr}. It follows that the first column is exact as well, and moreover it is a split exact sequence, because
\[ \Ext\left(\mc{O}_{\PP^1}(-q),\mc{O}_{\PP^1}(r-q)\right) = H^1(\PP^1,\mc{O}_{\PP^1}(r))=0.\]
Since $\mc{R}^{\oo q}=L_{q,q}(-q)$, conclusion (3) of our theorem is then a consequence of the $T$-isomorphisms
\[ D^{d-q}U\oo L_{q,q} \simeq H_{d,q}\quad\text{and}\quad T_qD^{d+q-r}U = \left(\bw^2 U\right)^{\oo(d-r+1)}\oo \Sym^{q-2+r-d}U\simeq H_{q-1,d-r+1},\]
where the second identification follows from \cite{gao-raicu}*{(2.6)} or \cite{sun}*{Proposition~3.5}. \qed 

Case (3) of Theorem~\ref{thm:recursive-Fdr} implies the following vanishing that will be used in the proof of case (4) later on.

\begin{corollary}\label{cor:vanish-H1Fdr}
 If $r\leq d<q$ then $H^1\left(\PP^1,\mc{F}^d_r(q-r)\right)=0$.
\end{corollary} 

\begin{proof}
 Using \eqref{eq:def-seq-Fdr} and the long exact sequence in cohomology, we get
 \begin{equation}\label{eq:H1Fdr=coker}
 H^1\left(\PP^1,\mc{F}^d_r(q-r)\right) = \op{coker}\left(D^d U \oo \Sym^{q-r}U \overset{\Delta}{\lra} D^{d-r}U \oo \Sym^q U\right).
 \end{equation}
To prove the desired vanishing, it suffices to show that the dual of the map in \eqref{eq:H1Fdr=coker} is injective. Using $(\Sym^i U)^{\vee} = D^i(U^{\vee})$, $(D^i U)^{\vee} = \Sym^i(U^{\vee})$, we can identify $U^{\vee} \simeq \kk^2 \simeq U$ and write the dual of \eqref{eq:H1Fdr=coker} as
\[H^1\left(\PP^1,\mc{F}^d_r(q-r)\right)^{\vee} \simeq \op{ker}\left(D^q U \oo \Sym^{d-r}U \overset{\Delta}{\lra} D^{q-r}U \oo \Sym^d U\right) = H^0\left(\PP^1,\mc{F}^q_{r}(d-r)\right).\]
Using Theorem~\ref{thm:recursive-Fdr}(3) we get
\[ \mc{F}^q_{r}(d-r) = H_{q,q}(d-q-r) \bigoplus H_{q-1,q-r+1}(d-q).\]
Our hypothesis implies $d-q-r\leq d-q < 0$, hence $H^0\left(\PP^1,\mc{F}^q_{r}(d-r)\right)=0$, concluding our proof.
\end{proof}

\subsection{Action of differential operators and consequences.} In this section we explore some consequences of the fact that the local cohomology modules $M$ in Section~\ref{sec:prelim} have a natural structure of $\mc{D}_R$-modules, where $\mc{D}_R$ is the ring of differential operators on $R$. Specifically we are interested in the action of differential operators
\[ \pd^{(a_1,a_2)} = \pd_1^{(a_1)}\pd_2^{(a_2)},\quad\text{ where }\pd_i^{(a)} = \frac{1}{a!}\cdot\frac{\pd^a}{\pd y_i^a},\]
which is $S$-linear. We will disregard differentiation with respect to the $x$ variables, so the notation $\pd_i$ should not be ambiguous. The following identity will be used repeatedly in what follows:
\begin{equation}\label{eq:pda1a2-act-ymon}
\pd^{(a_1,a_2)}\left(\frac{1}{y_1^{1+d_1}y_2^{1+d_2}} \right) = (-1)^{a_1+a_2}\cdot{d_1+a_1 \choose a_1}\cdot{d_2+a_2 \choose a_2}\cdot\frac{1}{y_1^{1+d_1+a_1}y_2^{1+d_2+a_2}}
\end{equation}

\begin{lemma}\label{lem:del-a1-a2-ops}
    If $a=a_1+a_2$, then the differential operators $\pd^{(a_1,a_2)}$ induce $S$- (respectively $\mc{O}_{\PP^1}$-) linear maps
    \[\pd^{(a_1,a_2)}:F^d_r \lra F^{d+a}_{r+a},\qquad \pd^{(a_1,a_2)}:\mc{F}^d_r \lra \mc{F}^{d+a}_{r+a}.\]
\end{lemma}

\begin{proof}
    We prove the $S$-module statement, which implies the sheaf statement by passing to the associated sheaves. It follows from \eqref{eq:pda1a2-act-ymon} that
    \[ \pd^{(a_1,a_2)}(M_{d,\bullet}) \subseteq M_{d+a,\bullet}.\]
    Consider now an element $f\in F^d_r$, that is, $f\in M_{d,\bullet}$ with $\omega^r\cdot f=0$. Using the fact that $\pd^{(a_1,a_2)}$ is a differential operator of order $a$, we can write
    \[ \omega^{r+a} \cdot \pd^{(a_1,a_2)} = Q \cdot \omega^r,\]
    for some differential operator $Q$. It follows that
    \[ \omega^{r+a} \cdot \pd^{(a_1,a_2)}(f) = Q \cdot \omega^r(f) = 0,\]
    hence $\pd^{(a_1,a_2)}(f)\in F^{d+a}_{r+a}$, as desired.
\end{proof}

\begin{lemma}\label{lem:Lucas-conseq}
    If $q=p^e$, $q'=p^{e-1}$, and if $v,f$ are non-negative integers satisfying $0\leq v< p$, $vq'\leq f <q$, then the binomial coefficient ${f \choose vq'}$ is non-zero (a unit) in $\kk$.
\end{lemma}

\begin{proof}
 If we write $f = uq'+f'$ with $0\leq f'<q'$, then we have $v\leq u<p$ and by Lucas' theorem it follows that
 \[ {f \choose vq'} \equiv {u\choose v} \ \text{mod }p.\]
 Since $v\leq u<p$, this implies that ${f \choose vq'}$ is non-zero in $\kk$.
\end{proof}

As a consequence of the results and arguments established above, we get the following consequence that will play an important role in part (4) of Theorem~\ref{thm:recursive-Fdr}.

\begin{corollary}\label{cor:Fq-Da-action-Fdr}
    Suppose that $q,q'$ are as in Lemma~\ref{lem:Lucas-conseq}, and that $a,r,d$ are integers satisfying $0\leq a<p$, $aq'\leq r\leq(a+1)q'$ and $r\leq d<q$. There exists a $T$-equivariant locally split inclusion
    \[\a:F^{q'}(H_{a,0}) \oo \mc{F}^{d-aq'}_{r-aq'} \lra \mc{F}^d_r,\]
    induced by Lemma~\ref{lem:del-a1-a2-ops} and the identification of $F^{q'}(H_{a,0})$ with the $\kk$-linear span of $\pd^{(a_1q',a_2q')}$, $a_1+a_2=a$.
\end{corollary}

\begin{proof}
    Since $\{y_1,y_2\}$ is the dual basis to $\{x_1,x_2\}$, and since the derivations $\{\pd_1,\pd_2\}$ are dual to $\{y_1,y_2\}$, it follows that $T$-action on $U$ naturally induces an isomorphism $\kk\cdot\pd^{(u,v)} \simeq L_{u,v}$. This implies that the $\kk$-linear span of $\pd^{(a_1q',a_2q')}$, $a_1+a_2=a$, is $T$-isomorphic to $F^{q'}(H_{a,0})$. Via this identification, the map
    \begin{equation}\label{eq:def-Fq'Ha0-action}
    \a:F^{q'}(H_{a,0}) \oo \mc{F}^{d-aq'}_{r-aq'} \lra \mc{F}^d_r,\quad \pd^{(a_1q',a_2q')}\oo f \lra \pd^{(a_1q',a_2q')}(f)
    \end{equation}
    is $T$-equivariant, so it suffices to check that it is a locally split inclusion. We do so on the affine chart $D_+(x_1)$, while the analogous verification on $D_+(x_2)$ follows by symmetry. To that end, we consider as in Lemma~\ref{lem:local-basis-Fdr} the $A$-basis of $\mc{F}^{d-aq'}_{r-aq'}(D_+(x_1))$ given by 
    \[ f_i' = \frac{1}{y_1^{1+i}y_2^{1+d-aq'-i}}+(\text{lower terms}),\ i=0,\cdots,r-aq'-1.\]
    It follows from \eqref{eq:pda1a2-act-ymon} that for $a_1+a_2=a$ and for $i=0,\cdots,r-aq'-1$ one has
    \[ \pd^{(a_1q',a_2q')}(f_i') = (-1)^{aq'}\cdot{i+a_1q' \choose a_1q'}\cdot{d-a_1q'-i \choose a_2q'}\cdot\frac{1}{y_1^{1+i+a_1q'}y_2^{1+d-a_1q'-i}}+(\text{lower terms}).\]
    Since $d,r<q$, it follows from Lemma~\ref{lem:Lucas-conseq} that each of the leading coefficients $(-1)^{aq'}\cdot{i+a_1q' \choose a_1q'}\cdot{d-a_1q'-i \choose a_2q'}$ is a unit in $\kk$. Moreover, since $r-aq'\leq q'$, the leading monomials $\frac{1}{y_1^{1+i+a_1q'}y_2^{1+d-a_1q'-i}}$ are also distinct as $i,a_1$ vary with $0\leq i\leq r-aq'-1$ and $0\leq a_1\leq a$. It follows that \eqref{eq:def-Fq'Ha0-action} is injective, and the sections $\pd^{(a_1q',a_2q')}(f_i')$ generate a free $A$-module which is a direct summand of $D^dU\oo A$. It follows that on $D_+(x_1)$ the composition
    \[ F^{q'}(H_{a,0}) \oo \mc{F}^{d-aq'}_{r-aq'} \lra \mc{F}^d_r \lra D^dU\oo\mc{O}_{\PP^1}\]
    is a split inclusion, which forces \eqref{eq:def-Fq'Ha0-action} to also be a split inclusion on $D_+(x_1)$.
\end{proof}

\begin{corollary}\label{cor:F_q'-vs-F_aq'}
    If $q'=p^{e-1}$, $q=p^e$, $1\leq a<p$ and $aq'\leq d<q$ then there is an isomorphism
    \[ \a':F^{q'}(H_{a-1,0}) \oo \mc{F}^{d-(a-1)q'}_{q'} \lra \mc{F}^d_{aq'}\]
    induced by differentiation as in Corollary~\ref{cor:Fq-Da-action-Fdr}. 
\end{corollary}

\begin{proof}
    We take $r=aq'$, so that $(a-1)q'\leq r\leq aq'$ is satisfied, hence we can apply Corollary~\ref{cor:Fq-Da-action-Fdr} in order to conclude that the map
    \[\a':F^{q'}(H_{a-1,0}) \oo \mc{F}^{d-(a-1)q'}_{r-(a-1)q'} \lra \mc{F}^d_r\] 
    is a locally split inclusion. Since both the source and the target are locally free of rank $aq'$, it follows that $\a'$ is an isomorphism.
\end{proof}

\subsection{The proof of part (4) of Theorem~\ref{thm:recursive-Fdr}.} It will be useful to restrict the possible $T$-equivariant summands that appear in $\mc{F}^d_r$. We show the following (see also \eqref{eq:char-0-Tequiv-Fdr} for the characteristic $0$ statement).

\begin{lemma}\label{lem:weights-Fdr} Let $d_0,r_0\in\bb{Z}_{\geq 0}\cup\{\infty\}$. If the isomorphisms in Theorem~\ref{thm:recursive-Fdr} hold for $d\leq d_0$, $r\leq r_0$, then for all such $d,r$ we have that every $T$-equivariant summand $L_{u,v}(i)$ of $\mc{F}^d_r$ satisfies $d-r+1\leq u,v\leq d$. 
\end{lemma}

\begin{proof}
    We argue by induction on $(d,r)$, considering each of the four cases from Theorem~\ref{thm:recursive-Fdr}. 
    
    In case (1) we have that $L_{u,v}$ is a $T$-equivariant summand of $H_{d,0}$, hence $u,v\geq 0$ and $u+v=d$. Since $d-r+1\leq 0$, the desired conclusion follows.

    In case (2), we have that $L_{u-q,v-q}(i+q)$ is a $T$-equivariant summand of $\mc{F}^{d-q}_r$, hence by induction we have 
    \[d-q-r+1\leq u-q,v-q\leq d-q,\] 
    which implies the desired inequalities for $u,v$.

    In case (3) we have two possibilities:
    \begin{itemize}
        \item If $L_{u,v}$ is a $T$-equivariant summand of $H_{d,q}$ then $u+v=d+q$, and $u,v\geq q\geq d-r+1$. In particular, 
        \[u,v \leq (u+v)-q = d,\] 
        as desired.
        \item If $L_{u,v}$ is a $T$-equivariant summand of $H_{q-1,d-r+1}$ then $u,v\geq d-r+1$ and $u+v=d+q-r$, hence
        \[u,v \leq (u+v)-(d-r+1) = q-1 \leq d,\] 
        as desired. 
    \end{itemize}
    In case (4) we have again two possibilities:
    \begin{itemize}
     \item If $L_{u,v}(i)$ is a $T$-equivariant summand of $F^{q'}(H_{a,0}) \oo \mc{F}^{d-aq'}_{r-aq'}$, then there exist $u_0,v_0\geq 0$, $u_0+v_0=a$, such that $L_{u-u_0q',v-v_0q'}(i)$ is a $T$-equivariant summand of $\mc{F}^{d-aq'}_{r-aq'}$. The lower bound $u,v\geq d-r+1=(d-aq')-(r-aq')+1$ follows by induction. For the upper bound, we have again by induction
     \[ u = (u-u_0q') + u_0q' \leq (d-aq')+u_0q' \leq d,\]
     and similarly $v\leq d$, as desired.
     \item If $L_{u,v}(i)$ is a $T$-equivariant summand of $F^{q'}(H_{a-1,0}) \oo \mc{F}^{d-r+q'}_{(a+1)q'-r}(r-aq')$, then there exist $u_0,v_0\geq 0$, $u_0+v_0=a-1$, such that $L_{u-u_0q',v-v_0q'}(i+aq'-r)$ is a $T$-equivariant summand of $\mc{F}^{d-r+q'}_{(a+1)q'-r}$. For the upper bound, we obtain using induction and the inequality $r\geq aq'$ that
    \[ u = (u-u_0q') + u_0q' \leq (d-r+q')+u_0q' \leq d-r+aq' \leq d,\]
    and similarly $v\leq d$. For the lower bound, we have
    \[ u \geq (u-u_0q') \geq (d-r+q') - ((a+1)q'-r) + 1 = d-aq'+1\geq d-r+1,\]
    and similarly $v\geq d-r+1$, concluding our proof.\qedhere
    \end{itemize}
\end{proof}

We next identify some trivial summands of the sheaf $\mc{F}^d_r$ under the hypothesis of Corollary~\ref{cor:Fq-Da-action-Fdr}. For $d_1,d_2\geq 0$ and $q'$ as above, we write
\begin{equation}\label{eq:d12-divrem-byq}
 d_1=u_1q'+v_1, \quad d_2=u_2q'+v_2,\quad\text{with }0\leq v_1,v_2<q',
\end{equation}
and note that if $d=d_1+d_2$ satisfies $d\geq aq'$ then $u_1+u_2\geq a-1$, and the equality $u_1+u_2=a-1$ may hold when $d\leq (a+1)q'-2$.

\begin{lemma}\label{lem:triv-summand-Fdr}
 Suppose that $aq'\leq d\leq (a+1)q'-2$ for some $1\leq a<p$, and define the subspace $W\subseteq D^dU$ given by
 \[ W = \op{Span}_{\kk}\left\{\left.\frac{1}{y_1^{1+d_1}y_2^{1+d_2}} \right| u_1+u_2=a-1 \right\},\]
 where $u_1,u_2$ are determined by $d_1,d_2$ via \eqref{eq:d12-divrem-byq}. We have an isomorphism $W\simeq F^{q'}(H_{a-1,0}) \oo H_{q'-1,d-aq'+1}$ which is $T$-equivariant, and moreover
 \[ \mc{W}^d_r := W\oo\mc{O}_{\PP^1}\text{ is a direct summand of }\mc{F}^d_r\quad\text{ if }r\geq aq'.\]
\end{lemma}

\begin{proof}
 The $T$-equivariant description of $W$ follows (recalling the convention \eqref{eq:forget-bw2Udual}) from the identities
 \[ \bigoplus_{u_1+u_2=a-1} L_{u_1q',u_2q'} = F^{q'}(H_{a-1,0})\quad\text{ and }\quad \bigoplus_{\substack{0\leq v_1,v_2\leq q'-1 \\ v_1+v_2 = d-(a-1)q'}} L_{v_1,v_2} = H_{q'-1,d-aq'+1}.\]

It is clear that $W\oo\mc{O}_{\PP^1}$ is a direct summand of $D^dU\oo\mc{O}_{\PP^1}$, so to show that $\mc{W}^d_r$ is a summand of~$\mc{F}^d_r$, it suffices to check that $\omega^r$ annihilates $W$. Since
\[ \omega^r = \omega^{r-aq'} \cdot \omega^{aq'} = \omega^{r-aq'} \cdot \left(x_1^{q'}y_1^{q'} + x_2^{q'}y_2^{q'} \right)^a,\]
it follows that each term in the monomial expansion of $\omega^r$ is divisible by some $y_1^{a_1q'}y_2^{a_2q'}$ with $a_1+a_2=a$. It follows that if $u_1+u_2=a-1$ then for each such $a_1,a_2$, either $a_1>u_1$ or $a_2>u_2$, which in turn implies $d_1-a_1q'<0$ or $d_2-a_2q'<0$. We conclude that $y_1^{a_1q'}y_2^{a_2q'}\cdot W=0$, hence $\omega^r\cdot W=0$, as desired.
\end{proof}

Notice that the isomorphism in Corollary~\ref{cor:F_q'-vs-F_aq'} restricts to an isomorphism
\[ \alpha' : F^{q'}(H_{a-1,0}) \oo \mc{W}^{d-(a-1)q'}_{q'} \lra \mc{W}^d_{aq'}.\]

\begin{lemma}\label{lem:twisted-Euler}
    Let $q,q'$ be as defined in Lemma~\ref{lem:Lucas-conseq}, let $d'=d-(a-1)q'$, and assume that $(a+1)q'-1\leq d<q$. We define
    \[ E^{(q')} =\pd^{(q',0)}y_1^{q'}+\pd^{(0,q')}y_2^{q'},\]
    and consider $Q=\pd^{(a_1'q',a_2'q')}$, with $a_1'+a_2'=a-1$. For every $f\in D^{d'}U$, we have
    \[ Q(E^{(q')}(f)) = E^{(q')}(Q(f)) + (a-1)Q(f).\]
\end{lemma}

\begin{proof}
 We first note that if $d_1,d_2$ are as in \eqref{eq:d12-divrem-byq} and $d_1+d_2<q$, then
 \begin{equation}\label{eq:Eqpr-action}
  E^{(q')}\left(\frac{1}{y_1^{1+d_1}y_2^{1+d_2}}\right) \overset{\eqref{eq:pda1a2-act-ymon}}{=} (-1)^{q'}\cdot\left[{d_1\choose q'} + {d_2\choose q'} \right] \frac{1}{y_1^{1+d_1}y_2^{1+d_2}} = (-u_1-u_2)\frac{1}{y_1^{1+d_1}y_2^{1+d_2}}
 \end{equation}
where the last equality uses the fact that $(-1)^{q'}=-1$ in $\kk$, and Lucas' theorem. To prove the lemma, it suffices to consider 
\[f=\frac{1}{y_1^{1+d'_1}y_2^{1+d'_2}},\text{ where }d'_1+d'_2=d'.\] 
If we write $d'_i = u'_iq'+v'_i$, $d_i = d'_i+a'_iq'$ and use \eqref{eq:d12-divrem-byq}, we get $u_i=u'_i+a'_i$, and the hypothesis $d<q$ implies that $u_i<p$. We can then use \eqref{eq:pda1a2-act-ymon} and Lucas' theorem to obtain
\[ Q(f)= (-1)^{(a_1'+a_2')q'}{d'_1+a'_1q'\choose a'_1q'}\cdot{d'_2+a'_2q'\choose a'_2q'}\cdot \frac{1}{y_1^{1+d'_1+a'_1q'}y_2^{1+d'_2+a'_2q'}} = (-1)^{a-1}{u_1\choose a'_1}\cdot{u_2\choose a'_2}\cdot \frac{1}{y_1^{1+d_1}y_2^{1+d_2}}.\]
Combining the above with \eqref{eq:Eqpr-action} it follows that
\[ Q(E^{(q')}(f)) = (-u'_1-u'_2)\cdot Q(f) =(-1)^a(u'_1+u'_2)\cdot{u_1\choose a'_1}\cdot{u_2\choose a'_2}\cdot \frac{1}{y_1^{1+d_1}y_2^{1+d_2}}.\]
We next compute, using the same notation and reasoning as above,
\[ E^{(q')}(Q(f)) \overset{\eqref{eq:pda1a2-act-ymon}}{=} (-1)^{a-1} {u_1\choose a'_1}\cdot{u_2\choose a'_2}\cdot E^{(q')}\left[\frac{1}{y_1^{1+d_1}y_2^{1+d_2}}\right] \overset{\eqref{eq:Eqpr-action}}{=} (-1)^a(u_1+u_2)\cdot{u_1\choose a'_1}\cdot{u_2\choose a'_2}\cdot \frac{1}{y_1^{1+d_1}y_2^{1+d_2}}.\]
The desired conclusion follows from the identity 
\[(-1)^a(u'_1+u'_2) = (-1)^a(u_1+u_2) + (-1)^{a-1}(a-1),\] 
which in turn holds since $(u_1-u'_1)+(u_2-u_2') = a'_1+a'_2=a-1$.
\end{proof}

We are now ready to verify the last case of Theorem~\ref{thm:recursive-Fdr}.

\begin{proof}[Proof of Theorem~\ref{thm:recursive-Fdr}(4)] We fix $d,r$ and assume by induction that Theorem~\ref{thm:recursive-Fdr}(4) holds for all smaller values of the parameters. Consider first the case when $d\leq (a+1)q'-2$. The filtration \cite{gao-raicu}*{(2.7)} reduces to a short exact sequence
\[ 0 \lra F^{q'}(D^{a-1}U) \oo T_{q'}D^{d-(a-1)q'}U \lra D^d U \lra F^{q'}(D^a U) \oo D^{d-aq'}U  \lra 0.\]
Using the $T$-equivariant identifications $D^{a-1}U = H_{a-1,0}$ and $T_{q'}D^{d-(a-1)q'}U=H_{q'-1,d-aq'+1}$, it follows that $F^{q'}(D^{a-1}U) \oo T_{q'}D^{d-(a-1)q'}U=W$, where $W$ is as defined in Lemma~\ref{lem:triv-summand-Fdr}. It follows that the map $\Delta$ in \eqref{eq:def-seq-Fdr} factors through $F^{q'}(D^a U) \oo D^{d-aq'}U$, and we can write
\[ \mc{F}^d_r = \mc{W}^d_r \oplus \ol{\mc{F}}^d_r,\quad\text{where}\quad \ol{\mc{F}}^d_r = \ker\left(F^{q'}(D^a U) \oo D^{d-aq'}U\oo\mc{O}_{\PP^1} \overset{\ol{\Delta}}{\lra} D^{d-r}U(r) \right)\]
We can further factor $\ol{\Delta}$ as
\[ F^{q'}(D^a U) \oo D^{d-aq'}U\oo\mc{O}_{\PP^1} \overset{\ol{\Delta}'}{\lra} F^{q'}(D^a U) \oo D^{d-r}U(r-aq') \overset{\ol{\Delta}''}{\lra} D^{d-r}U(r), \]
where $\ol{\Delta}''$ is the surjection induced by
\[ D^a U \oo \mc{O}_{\PP^1} = \Sym^a U \oo \mc{O}_{\PP^1} \onto \mc{O}_{\PP^1}(a)\]
whose kernel is $\bw^2 U \oo \Sym^{a-1}U \oo \mc{O}_{\PP^1}(-1) = H_{a,1}(-1)$, and $\ol{\Delta}'$ is obtained by tensoring with $F^{q'}(D^a U)=F^{q'}(H_{a,0})$ from the surjection defining the short exact sequence (see \eqref{eq:def-seq-Fdr})
\[ 0 \lra \mc{F}^{d-aq'}_{r-aq'} \lra D^{d-aq'}U\oo\mc{O}_{\PP^1} \lra D^{d-r}U(r-aq') \lra 0.\]
Since $\ker(\ol{\Delta})$ is an extension of $\ker(\ol{\Delta}'')$ by $\ker(\ol{\Delta}')$, it follows that we have a short exact sequence
\begin{equation}\label{eq:ses-Fbardr}
 0 \lra F^{q'}(H_{a,0})\oo \mc{F}^{d-aq'}_{r-aq'} \lra \ol{\mc{F}}^d_r  \lra F^{q'}(H_{a,1}) \oo H_{d-r,0}(r-(a+1)q') \lra 0
\end{equation}
Using the reinterpretation of Theorem~\ref{thm:recursive-Fdr}(4) as in Corollary~\ref{cor:rec-Fdr-sometimes}(i), along with the identification $F^{q'}(H_{a,1}) \oo H_{d-r,0}=F^{q'}(H_{a-1,0}) \oo H_{d-r+q',q'}$, it suffices to verify that \eqref{eq:ses-Fbardr} splits. This is further equivalent to showing
\[ \Ext^1\left(\mc{O}_{\PP^1}(r-(a+1)q'),\mc{F}^{d-aq'}_{r-aq'}\right) = H^1\left(\PP^1,\mc{F}^{d-aq'}_{r-aq'}((a+1)q'-r)\right) = 0,\]
which follows from Corollary~\ref{cor:vanish-H1Fdr} using the inequalities $r-aq'\leq d-aq' < q'$.

It remains to consider the case when $(a+1)q'-1\leq d\leq q-1$. We will construct below a $T$-equivariant exact sequence
\begin{equation}\label{eq:case4-Frd-seq} 
0 \lra F^{q'}(H_{a,0}) \oo \mc{F}^{d-aq'}_{r-aq'} \overset{\a}{\lra} \mc{F}^d_r \lra F^{q'}(H_{a-1,0}) \oo \mc{F}^{d-r+q'}_{(a+1)q'-r}(r-aq') \lra 0.
\end{equation}
To prove that \eqref{eq:case4-Frd-seq} is split, it will then suffice to show that there are no $T$-invariants in
\[ \Ext^1\left(F^{q'}(H_{a-1,0}) \oo \mc{F}^{d-r+q'}_{(a+1)q'-r}(r-aq'),F^{q'}(H_{a,0}) \oo \mc{F}^{d-aq'}_{r-aq'} \right).\]
Suppose otherwise, and consider $T$-equivariant summands $L_{u_1,v_1}(i_1)$ of $\mc{F}^{d-r+q'}_{(a+1)q'-r}$, and $L_{u_2,v_2}(i_2)$ of $\mc{F}^{d-aq'}_{r-aq'}$, such that there are non-zero $T$-invariants in
\[ \Ext^1\left(F^{q'}(H_{a-1,0}) \oo L_{u_1,v_1}(i_1+r-aq'),F^{q'}(H_{a,0}) \oo L_{u_2,v_2}(i_2) \right),\]
or equivalently, there are non-zero $T$-invariants in
\[ F^{q'}\left(H_{a,0}\oo H_{a-1,0}^{\vee}\right) \oo L_{u_2-u_1,v_2-v_1} \oo H^1\left(\PP^1,\mc{O}_{\PP^1}(i_2-i_1-r+aq')\right).\]
By Lemma~\ref{lem:weights-Fdr}, we have that 
\[u_2-u_1,v_2-v_1 \leq (d-aq') - \left((d-r+q')-((a+1)q'-r)+1 \right) = -1.\]
Since every $T$-equivariant summand $L_{u_3,v_3}$ of $H^1\left(\PP^1,\mc{O}_{\PP^1}(i_2-i_1-r+aq')\right)$ satisfies $u_3,v_3\leq -1$, it follows that for some $(u_3,v_3)$ there exists a summand of $F^{q'}\left(H_{a,0}\oo H_{a-1,0}^{\vee}\right)$ of the form $L_{u,v}$, where $u=u_1-u_2-u_3>0$ and $v=v_1-v_2-v_3>0$. We must then have that $q'|u,v$, and 
\[ 0 < u+v = aq'-(a-1)q' = q',\]
which is impossible.

To construct \eqref{eq:case4-Frd-seq} we let $d'=d-(a-1)q'$, $r'=r-aq'$, and build a commutative diagram
\begin{equation}\label{eq:comm-square-part4}
\begin{gathered}
\xymatrix{
 F^{q'}(H_{a-1,0}) \oo \mc{F}^{d'}_{r'} \ar[r]^(.7){\a'} \ar[d]_{\b'} & \mc{F}^d_{aq'} \ar[d]^{\b} \\
 F^{q'}(H_{a,0}) \oo \mc{F}^{d'-q'}_{r'} \ar[r]_(.7){\a} & \mc{F}^d_{r} \\
}
\end{gathered}
\end{equation}
where we think of $F^{q'}(H_{a-1,0})$ and $F^{q'}(H_{a,0})$ as spaces of differential operators as in Corollary~\ref{cor:Fq-Da-action-Fdr}. More precisely, we define $\a$ as in Corollary~\ref{cor:Fq-Da-action-Fdr} (note that $d'-q'=d-aq'$), which is a locally split inclusion because $d<q$. We define $\a'$ as the composition of the inclusion $\mc{F}^{d'}_{r'}\subseteq\mc{F}^{d'}_{q'}$ with the isomorphism from Corollary~\ref{cor:F_q'-vs-F_aq'}, and note that using \eqref{eq:ses-Fdrs} we have $\coker(\a') \simeq F^{q'}(H_{a-1,0})\oo\mc{F}^{d'-r'}_{q'-r'}(r')$. Using \eqref{eq:mult-yi-Fdr} and the inclusion $\mc{F}^d_{aq'}\subseteq\mc{F}^d_r$, we define $\beta'$, $\beta$ via
\[ \beta'(Q \oo f) = Q\cdot\pd^{(q',0)} \oo y_1^{q'}\cdot f + Q\cdot\pd^{(0,q')} \oo y_2^{q'}\cdot f,\qquad\text{and}\qquad \beta(g) = E^{(q')}(g) + (a-1)g.\]
The commutativity of \eqref{eq:comm-square-part4} follows from Lemma~\ref{lem:twisted-Euler}.

The existence of the short exact sequence \eqref{eq:case4-Frd-seq} now reduces to proving that $\coker(\a)\simeq\coker(\a')$. Notice that $\coker(\a),\coker(\a')$ are locally free sheaves of the same rank, namely
\[ a(q'-r') = r - (a+1)r',\]
so it suffices to check that the natural map $\coker(\a')\lra\coker(\a)$ is surjective. This in turn reduces to proving that $\op{Im}(\a) + \op{Im}(\beta) = \mc{F}^d_r$. We check this locally on the chart $D_+(x_1)$, since the analogous statement on $D_+(x_2)$ follows by symmetry. To that end, it suffices to show that for each $k=0,\cdots,r-1$, $\op{Im}(\a) + \op{Im}(\beta)$ contains a section of $\mc{F}^d_r(D_+(x_1))$ of the form
\[ g_k = \text{unit}\cdot\frac{1}{y_1^{1+k}y_2^{1+d-k}}+(\text{lower terms}).\]
For $k = 0,\cdots,aq'-1$, we can take $g_k = \beta(f_k)$, where $\{f_0,\cdots,f_{aq'-1}\}$ is the $A$-basis of $\mc{F}^d_{aq'}(D_+(x_1))$ constructed in Lemma~\ref{lem:local-basis-Fdr} (in the special case $r=aq'$). To see that $g_k$ has the desired form, we set $d_1=k$, $d_2=d-k$ in \eqref{eq:Eqpr-action} to conclude that
\[ g_k = (a-1-u_1-u_2)\cdot\frac{1}{y_1^{1+k}y_2^{1+d-k}}+(\text{lower terms})\]
The assumption $d\geq(a+1)q'-1$ forces $u_1+u_2\geq a$, while $d<q$ implies $u_1+u_2<p$. We conclude that $a-1-p < a-1-u_1-u_2 \leq -1$, that is, $(a-1-u_1-u_2)$ is a unit in $\kk$.

For $k=aq'+i$ with $0\leq i\leq r'-1$, we can take 
\[ g_k = \a\left(\pd^{(aq',0)}\oo f_i'\right) = \pd^{(aq',0)}(f_i') = (-1)^{aq'}{k\choose aq'}\cdot\frac{1}{y_1^{1+k}y_2^{1+d-k}}+(\text{lower terms}),\]
where $f_i'$ are as defined in the proof of Corollary~\ref{cor:Fq-Da-action-Fdr}, and ${k\choose aq'}$ is a unit by Lemma~\ref{lem:Lucas-conseq}.
\end{proof}

\section{Recursive structure for the cohomology of $\mc{F}^d_r$}
\label{sec:coh-Fdr-recursive}

The goal of this section is to establish some consequences of Theorem~\ref{thm:recursive-Fdr} and describe the $T$-equivariant structure of the cohomology of (twists of) the sheaves $\mc{F}^d_r$ as the parameters $d,r$ vary. The results here will then be employed in Section~\ref{sec:Han-Monsky}. We focus on two short exact sequences arising as special cases of \eqref{eq:ses-Fdrs}
\begin{equation}\label{eq:ses-Fdrs-r=1} 
0 \lra L_{d,d}(-d) \lra \mc{F}^d_{r+1} \lra \mc{F}^{d-1}_r(1) \lra 0
\end{equation}
\begin{equation}\label{eq:ses-Fdrs-s-r=1} 
0 \lra \mc{F}^d_r \lra \mc{F}^d_{r+1} \lra L_{d-r,d-r}(2r-d) \lra 0
\end{equation}
and aim to quantify in both cases the failure of exactness of the global sections functor. 

We write $S=\kk[x_1,x_2]$ as before, so that $\PP^1=\op{Proj}(S)$, and consider the natural action of the $2$-dimensional torus $T$ from Section~\ref{sec:pparts-P1}. We will work exclusively in the categories of $T$-equivariant $S$-modules and $T$-equivariant sheaves on $\PP^1$ respectively. Every finitely generated $T$-equivariant $S$-module $M$ has a natural bi-grading, where $M_{u,v}$ is the isotypic component corresponding to the irreducible representation $L_{u,v}$ (see \eqref{eq:act-on-chiuv}), and an associated character defined as the Laurent power series
\[[M] = \sum_{u,v\in\bb{Z}} \dim(M_{u,v})\cdot z_1^u \cdot z_2^v \quad\text{ in }\bb{Z}((z_1,z_2)).\]
More generally, we use the notation $[M]$ to denote the character of a $T$-representation $M$ where all the isotypic components are finite dimensional. If $M$ has finite length (that is, $\dim_{\kk}(M)<\infty$) then we write $M^{\vee} = \Hom_{\kk}(M,\kk)$ with its induced $T$-equivariant structure. We note that in this case $[M]\in\bb{Z}[z_1^{\pm 1},z_2^{\pm 1}]$ is a Laurent polynomial, and we have $[M^{\vee}]=[M]^{\vee}$ where ${}^{\vee}:\bb{Z}[z_1^{\pm 1},z_2^{\pm 1}]\lra \bb{Z}[z_1^{\pm 1},z_2^{\pm 1}]$ denotes the involution sending $z_i \mapsto z_i^{-1}$. With the notation in Section~\ref{sec:pparts-P1} we have
\[ [H_{a,b}] = z_1^az_2^b+z_1^{a-1}z_2^{b+1}+\cdots + z_1^{b+1}z_2^{a-1}+z_1^bz_2^a=(z_1z_2)^b\cdot h_{a-b},\quad\text{ for }a\geq b\geq 0, \]
where $h_d=h_d(z_1,z_2)$ denotes the complete symmetric polynomial of degree $d$ in two variables. We also have
\[ \left[L_{u,v}\oo S \right] = \frac{z_1^uz_2^v}{(1-z_1)(1-z_2)}\quad\text{ for }u,v\in\bb{Z}.\]
We note that every $T$-equivariant $S$-submodule $M\subset L_{u,v}\oo S$ has the form $L_{u,v} \oo I$ for some monomial ideal~$I$, and that $M$ is uniquely determined by its character $[M]$.

We consider the cohomology functors $H^i_*(\PP^1,-) = \bigoplus_{e\in\bb{Z}}H^i_*(\PP^1,-(e))$, and note that
\[ H^0_*(\PP^1,\mc{F}^d_r) = F^d_r,\quad H^0_*(\PP^1,\mc{O}_{\PP^1}) = S,\quad H^1_*(\PP^1,\mc{O}_{\PP^1}) = H^1_{\mf{m}}(S),\]
where $\mf{m}=\langle x_1,x_2\rangle$ is the maximal homogeneous ideal in $S$. As a $T$-representation, we have
\[ H^1_{\mf{m}}(S) = \bigoplus_{u,v<0} L_{u,v},\]
and as an $S$-module it is Artinian (though not finitely generated) with simple socle given by the subrepresentation $L_{-1,-1} = \kk\cdot\frac{1}{x_1x_2}$. It follows that every finitely generated $T$-equivariant submodule $M\subseteq H^1_{\mf{m}}(S)$ has finite length, and it is of the form
\[ M = L_{-1,-1} \oo (S/I)^{\vee},\quad\text{ for $I$ an $\mf{m}$-primary monomial ideal}.\]
The beginning of the long exact sequence in cohomology associated to \eqref{eq:ses-Fdrs-r=1} is given by
\[ 0 \lra L_{d,d}\oo S \lra F^d_{r+1} \lra F^{d-1}_{r} \overset{\delta^d_r}{\lra} L_{d,d}\oo H^1_{\mf{m}}(S) \]
so the failure of exactness of $H^0_*$ can be measured by the image of the connecting homomorphism $\delta^d_r$. We let
\begin{equation}\label{eq:def-Idr-Bdr} \op{Im}(\delta^d_r) = L_{d-1,d-1} \oo B(d,r)^{\vee},\text{ where }B(d,r) = S/I(d,r)
\end{equation}
for some $\mf{m}$-primary monomial ideal $I(d,r)\subset S$. 

Similarly, the long exact sequence in cohomology for \eqref{eq:ses-Fdrs-s-r=1} yields
\[ 0 \lra F^d_r \lra F^d_{r+1} \overset{\gamma^d_r}{\lra} L_{d-r,d-r} \oo S\]
whose failure of exactness can be measured by the cokernel of $\gamma^d_r$. If $d\geq r$ then we write
\begin{equation}\label{eq:def-Jdr-Cdr} \op{coker}(\gamma^d_r) = L_{d-r,d-r} \oo C(d,r),\text{ where }C(d,r) = S/J(d,r)
\end{equation}
for some $\mf{m}$-primary monomial ideal $J(d,r)\subset S$. Note that by \eqref{eq:Fdr-for-d<r} we have $F^d_r = F^d_{r+1}$ when $d<r$, hence if we extended definition \eqref{eq:def-Jdr-Cdr} to $d<r$, it would imply $J(d,r) = 0$ in this range. Since this will be inconvenient for some of the arguments, we will be using a different convention for $d<r$, as explained in Remark~\ref{rem:periodicity-Jdr}.

\begin{remark}\label{rem:GL2-equivariance}
    The reader may notice that the ideals $I(d,r)$ and $J(d,r)$ are in fact $\GL(U)$-equivariant ($U=\kk^2$) since all the maps we considered so far are $\GL(U)$-equivariant if we identify $L_{d,d}$ with $(\bw^2 U)^{\oo d}$.
\end{remark}

We will use Theorem~\ref{thm:recursive-Fdr} to prove the following recursive relations for the ideals $I(d,r)$, $J(d,r)$ (or equivalently for the finite length quotients $B(d,r)$, $C(d,r)$).

\begin{theorem}\label{thm:recursions-Idr-Jdr}
    Suppose that $\op{char}(\kk)=p>0$, $q'\leq r< q$ where $q=p^e$, $q'=p^{e-1}$.
    \begin{enumerate}
        \item If $d\leq r$ then $I(d,r)=\mf{m}^d$.
        \item If $d\geq q+r$ then $J(d,r)=J(d-q,r)$.
        \item If $d\geq q$ then $I(d,r)=I(d-q,r)$. In particular, if $q\leq d\leq q+r$ then $I(d,r)=\mf{m}^{d-q}$.
        \item If $r\leq d\leq q+r-1$ then $J(d,r) = I(q+r-1-d,r)$. In particular, if $q-1\leq d\leq q+r-1$ then $J(d,r)=\mf{m}^{r+q-1-d}$.
        \item If $r\leq d<q$, and if $a$ is such that $aq'\leq r<(a+1)q'$, then
        \[[B(d,r)] = F^{q'}(h_a)\cdot[B(d-aq',r-aq')] - F^{q'}(h_{a-1})\cdot(z_1z_2)^{q'-1}\cdot[C(d-r-1+q',(a+1)q'-r-1)]^{\vee} + [S/\mf{m}^{[aq']}]\]
        \[[C(d,r)] = F^{q'}(h_a)\cdot[C(d-aq',r-aq')] - F^{q'}(h_{a-1})\cdot(z_1z_2)^{q'-1}\cdot[B(d-r+q',(a+1)q'-r-1)]^{\vee}+ [S/\mf{m}^{[aq']}]\]
    \end{enumerate}
\end{theorem}

The only unexplained notation in the theorem above is $\mf{m}^{[aq']}$ which stands for the (generalized) Frobenius power of $\mf{m}$: with our convention, $a<p$, and
\[\mf{m}^{[aq']} = \left(\mf{m}^a\right)^{[q']} = \langle f^{q'} : f\in \mf{m}^a\rangle. \]
We summarize some elementary character formulas that will be used in the proofs below. Using the fact that the minimal free resolution of $S/\mf{m}^{d}$ is given by (see for instance \cite{rai-surv}*{Theorem~2.2})
\[ 0 \lra H_{d,1}\oo S \lra H_{d,0}\oo S\lra S \lra S/\mf{m}^{d} \lra 0,\]
it follows by taking Euler characteristic that
\begin{equation}\label{eq:char-S-mod-md}
   [S/\mf{m}^{d}] = \frac{1-h_d+z_1z_2h_{d-1}}{(1-z_1)(1-z_2)}.
\end{equation}
Using the flatness of Frobenius, we get similarly a resolution
\[ 0 \lra F^{q'}(H_{a,1})\oo S \lra F^{q'}(H_{a,0})\oo S\lra S \lra S/\mf{m}^{[aq']} \lra 0,\]
which then implies
\[   [S/\mf{m}^{[aq']}] = \frac{1-F^{q'}(h_a)+(z_1z_2)^{q'} F^{q'}(h_{a-1})}{(1-z_1)(1-z_2)}.
\]
Using the identities
\begin{equation}\label{eq:ha-Fqha-dual}
 h_a^{\vee}=(z_1z_2)^{-a}\cdot h_a,\quad F^{q'}(h_a^{\vee}) = (z_1z_2)^{-aq'}\cdot F^{q'}(h_a)
\end{equation}
it follow that
\begin{equation}\label{eq:dual-S-mod-Frob-power}
   (z_1z_2)^{aq'-1}[S/\mf{m}^{[aq']}]^{\vee} = \frac{(z_1z_2)^{-1}\left((z_1z_2)^{aq'}-F^{q'}(h_a)+F^{q'}(h_{a-1})\right)}{(1-z_1^{-1})(1-z_2^{-1})} = \frac{(z_1z_2)^{aq'}-F^{q'}(h_a)+F^{q'}(h_{a-1})}{(1-z_1)(1-z_2)}.
\end{equation}

\begin{remark}\label{rem:periodicity-Jdr}
    It follows from part (2) of Theorem~\ref{thm:recursions-Idr-Jdr} that for $r<q$ and $d\geq r$, the ideal $J(d,r)$ depends only on the residue class of $d$ modulo $q$. It will then be convenient to extend the definition of $J(d,r)$ to $d<r$ by using this periodicity: if $d<r<q$ then we define $J(d,r) = J(\tilde{d},r)$ where $\tilde{d}$ is any positive integer satisfying $\tilde{d}\geq r$ and $q|\tilde{d}-d$. We note also that for this definition it is not necessary to assume that $q'\leq r<q$ (if $r<q'$ then $J(d,r)$ is in fact $q'$-periodic in $d$).
\end{remark}

\begin{proof}[Proof of Theorem~\ref{thm:recursions-Idr-Jdr}]
    For part (1), the hypothesis $d\leq r$ implies $F^d_{r+1} = H_{d,0}\oo S$, $F^{d-1}_r=H_{d-1,0}\oo S$, hence \eqref{eq:def-Idr-Bdr} implies
\[
    \begin{aligned}
{[L_{d-1,d-1} \oo B(d,r)^{\vee}]} &= [L_{d,d}\oo S] - [H_{d,0}\oo S] + [H_{d-1,0}\oo S] \\
&= \frac{(z_1z_2)^d-h_d+h_{d-1}}{(1-z_1)(1-z_2)}
    \end{aligned}
\]
    Dividing by $[L_{d-1,d-1}]=(z_1z_2)^{d-1}$, taking duals, and using \eqref{eq:ha-Fqha-dual} we get
    \[[B_{d,r}] = \frac{(z_1z_2)^{-1}(1-h_{d}+z_1z_2 h_{d-1})}{(1-z_1^{-1})(1-z_2^{-1})} = \frac{1-h_{d}+z_1z_2 h_{d-1}}{(1-z_1)(1-z_2)}.\]
Using \eqref{eq:char-S-mod-md}, we conclude that $[B(d,r)] = [S/\mf{m}^{d}]$, hence $I(d,r)=\mf{m}^{d}$.

    For part (2), we note that $q'\leq r,r+1\leq q$, hence we can apply Theorem~\ref{thm:recursive-Fdr}(2) to conclude that
    \[ F^d_r = L_{q,q} \oo F^{d-q}_r\quad\text{ and }F^d_{r+1} = L_{q,q} \oo F^{d-q}_{r+1} \quad\text{ if }d\geq q+r.\]
    It follows from \eqref{eq:def-Jdr-Cdr} that
    \[ 
    \begin{aligned}
    (z_1z_2)^{d-r}[C(d,r)] &= [L_{d-r,d-r}\oo C(d,r)] = [F^d_r] - [F^d_{r+1}] + [L_{d-r,d-r}\oo S] \\
    &= (z_1z_2)^q\cdot\left([F^{d-q}_r] - [F^{d-q}_{r+1}] + [L_{d-q-r,d-q-r}\oo S]\right) = (z_1z_2)^q \cdot (z_1z_2)^{d-q-r}[C(d-q,r)]
    \end{aligned}\]
    This implies that $C(d,r)=C(d-q,r)$ and hence $J(d,r)=J(d-q,r)$, as desired.

    For part (3), the same reasoning as in the previous paragraph implies that $I(d,r)=I(d-q,r)$ for $d\geq q+r$, so we may restrict to the case $q\leq d\leq q+r-1$. We can then apply Theorem~\ref{thm:recursive-Fdr}(3) to conclude that
    \[
    \begin{aligned}
{[L_{d-1,d-1} \oo B(d,r)^{\vee}]} &= [L_{d,d}\oo S] - [F^d_{r+1}] + [F^{d-1}_r] \\
&= \left([L_{d,d}]-[H_{d,q}]-[H_{q-1,d-r}] + [H_{d-1,q}]+[H_{q-1,d-r}] \right)\cdot[S] \\
&= \left([L_{d,d}]-[H_{d,q}] + [H_{d-1,q}] \right)\cdot[S] \\
&= \frac{(z_1z_2)^d-(z_1z_2)^q\cdot h_{d-q}+(z_1z_2)^q\cdot h_{d-q-1}}{(1-z_1)(1-z_2)}
    \end{aligned}
    \]
Arguing as in part (1), it follows that
    \[[B(d,r)] = \frac{(z_1z_2)^{-1}(1-h_{d-q}+z_1z_2 h_{d-q-1})}{(1-z_1^{-1})(1-z_2^{-1})} =[S/\mf{m}^{d-q}],\] 
hence $I(d,r)=\mf{m}^{d-q}$. Using part (1) together with the fact that $0\leq d-q<r$, we also get $I(d-q,r)=\mf{m}^{d-q}$, hence $I(d,r)=I(d-q,r)$, as desired. When $d=q+r$, we have $I(d,r)=I(r,r)=\mf{m}^r$ by part (1).

We first prove conclusion (4) in the case $q\leq d\leq q+r-1$. As in part (3), we can apply Theorem~\ref{thm:recursive-Fdr}(3) to conclude that
\[
\begin{aligned}
(z_1z_2)^{d-r}\cdot[C(d,r)] &= [F^d_r] - [F^d_{r+1}] + [L_{d-r,d-r}\oo S] \\
&= (z_1z_2)^{d-r}\cdot\left([H_{q+r-1-d,1}\oo S] - [H_{q+r-1-d,0}\oo S] + [S] \right) = (z_1z_2)^{d-r}\cdot[S/\mf{m}^{q+r-1-d}].
\end{aligned}
\]
This shows that $J(d,r)=\mf{m}^{q+r-1-d}$, which in turn equals $I(q+r-1-d,r)$ by part (1), since $q+r-1-d\leq r$.

To prove the remaining part of (4), as well as (5), we may then assume that $r\leq d\leq q-1$. We first prove~(5), which we then use to verify~(4). We have using \eqref{eq:def-Idr-Bdr}, \eqref{eq:def-Jdr-Cdr} and Theorem~\ref{thm:recursive-Fdr} that
\[
\begin{aligned}
(z_1z_2)^{d-1}\cdot [B(d,r)^{\vee}] =& (z_1z_2)^{d}\cdot[S] - [F^d_{r+1}] + [F^{d-1}_r] \\
=& F^{q'}(h_a)\cdot\left((z_1z_2)^{d-aq'}\cdot[S]-[F^{d-aq'}_{r+1-aq'}] + [F^{d-1-aq'}_{r-aq'}]\right) - \\
& F^{q'}(h_{a-1})\cdot\left((z_1z_2)^{d-aq'}\cdot[S]-[F^{d-r-1+q'}_{(a+1)q'-r}] + [F^{d-r-1+q'}_{(a+1)q'-r-1}]\right) + \\
& (z_1z_2)^{d-aq'}\cdot\left(F^{q'}(h_{a-1})-F^{q'}(h_a) + (z_1z_2)^{aq'}\right) \cdot [S] \\
\overset{\eqref{eq:dual-S-mod-Frob-power}}{=}& F^{q'}(h_a)\cdot(z_1z_2)^{d-aq'-1}[B(d-aq',r-aq')^{\vee}] - \\
& F^{q'}(h_{a-1})\cdot(z_1z_2)^{d-aq'}[C(d-r-1+q',(a+1)q'-r-1)] + \\
&(z_1z_2)^{d-1}\cdot[S/\mf{m}^{[aq']}]^{\vee}
\end{aligned}
\]
Dividing by $(z_1z_2)^{d-1}$, taking duals, and using \eqref{eq:ha-Fqha-dual} gives the desired formula for $[B(d,r)]$. The one for $[C(d,r)]$ follows using similar reasoning.

To finish part (4) of the theorem, we let $\tilde{d}=q+r-1-d$ and argue by induction. Using the fact that $r-aq'<q'$ we can apply the periodicity statements (2), (3) (with $q$ replaced by $q'$) and induction to conclude
\[ C(\tilde{d}-aq',r-aq') = B(q'+r-1-\tilde{d},r-aq') = B(d-(p-1)q',r-aq') = B(d-aq',r-aq').\]
Similarly, since $(a+1)q'-r-1 < q'$ we get 
\[ B(\tilde{d}-r+q',(a+1)q'-r-1) = C(d-r-1+q',(a+1)q'-r-1).\]
It follows from (5) that $B(d,r) = C(\tilde{d},r)$, hence $I(d,r)=J(\tilde{d},r)$, concluding the inductive step.
\end{proof}

Notice that by \eqref{eq:def-Idr-Bdr}, \eqref{eq:def-Jdr-Cdr} we have
\begin{equation}\label{eq:Idr-Jdr-r=0} 
 B(d,0) = C(d,0) = 0 \quad\text{and}\quad I(d,0) = J(d,0) = S.
\end{equation}
Moreover, it follows from Theorem~\ref{thm:recursions-Idr-Jdr}(5) (with $q'=1$, $q=p$, $a=r$), that
\begin{equation}\label{eq:Idr-Jdr-smallr} 
I(d,r) = J(d,r) = \mf{m}^r \quad\text{for }1\leq r\leq d<p.
\end{equation}

Our interest in Theorem~\ref{thm:recursions-Idr-Jdr} comes from the following simple consequence, that will be crucial for the results in Section~\ref{sec:Han-Monsky}. We note that by restricting the $T$-action to the diagonal subtorus, every $T$-equivariant module $M$ will inherit a $\bb{Z}$-grading, where $M_{u,v}$ lies in degree $u+v$. For a monomial ideal $I\subset S$, this yields the standard $\bb{Z}$-grading.

\begin{corollary}\label{cor:quot-Jdrs}
    If $p\nmid r$ then 
    \[\left(\frac{J(d-1,r-1)}{J(d,r)}\right)_i = 0\text{ for }i\neq r-1.\]
\end{corollary}

\begin{proof}
    We let $q',q$ are as in Theorem~\ref{thm:recursions-Idr-Jdr}, and we may assume using periodicity (Theorem~\ref{thm:recursions-Idr-Jdr}(2), Remark~\ref{rem:periodicity-Jdr}) that $r\leq d\leq q+r-1$. Consider first the case $r=1$, where we have by \eqref{eq:Idr-Jdr-smallr} and Theorem~\ref{thm:recursions-Idr-Jdr}(4) that
    \[ J(d,1) = \mf{m}\quad\text{for }1\leq d<p,\quad J(d,p) = S.\]
    It follows from \eqref{eq:Idr-Jdr-r=0} that
    \begin{equation}\label{eq:quot-Jdr-r=1}
    \frac{J(d-1,0)}{J(d,1)} = \begin{cases}
        S/\mf{m} & \text{if }1\leq d<p,\\
        0 & \text{if }d=p,
    \end{cases}
    \end{equation}
    which vanishes in degree $i\neq 0$.

    If $r>1$, then using the hypothesis $p\nmid r$ we conclude that $q'\leq r-1,r<q$. If $d\geq q$ then it follows from Theorem~\ref{thm:recursions-Idr-Jdr}(4) that $J(d-1,r-1)=J(d,r)=\mf{m}^{r+q-1-d}$, so the conclusion follows. We are therefore left with the case $r\leq d<q$. If $q'=1$ (so that $q=p$) then we have by \eqref{eq:Idr-Jdr-smallr}
    \[ \frac{J(d-1,r-1)}{J(d,r)} = \frac{\mf{m}^{r-1}}{\mf{m}^r},\]
    which is concentrated in degree $r-1$. We may therefore assume that $q'>1$, and in particular $p|q'$, and we apply Theorem~\ref{thm:recursions-Idr-Jdr}(3)--(5) and Remark~\ref{rem:periodicity-Jdr}. If we define
    \[d'=d-aq',\quad r'=r-aq',\quad d'' = (a+1)q'-1-d,\quad r'' = (a+1)q'-r\]
    then we get
    \begin{equation}\label{eq:quot-consec-Jdr}
    \left[\frac{J(d-1,r-1)}{J(d,r)}\right] = F^{q'}(h_a)\cdot \left[\frac{J(d'-1,r'-1)}{J(d',r')}\right] + F^{q'}(h_{a-1})\cdot(z_1z_2)^{q'-1}\cdot \left[\frac{J(d''-1,r''-1)}{J(d'',r'')}\right]^{\vee}
    \end{equation}
    Since $p\nmid r$ and $p|q'$, it follows that $p\nmid r',r''$. Using the fact that
    \[q'a+(r'-1) = r-1 = q'(a-1)+2(q'-1)-(r''-1)\]
    it follows by induction that the terms on the right side of \eqref{eq:quot-consec-Jdr} vanish in degrees $\neq r-1$, as desired.
\end{proof}

We end this section with an explicit computation of the quotients in Corollary~\ref{cor:quot-Jdrs} when $p=2$ (and $r$ odd). To do so, we define the \defi{no-carries set}
\[ NC = \left\{ (a,b) \in \bb{Z}^2_{\geq 0} : {a+b\choose a}\text{ is odd}\right\}\]
If $a=\sum 2^{a_i}$, $b = \sum 2^{b_j}$, are the $2$-adic expansions of $a,b$, then $(a,b)\in NC$ if and only if $a_i\neq b_j$ for all $i,j$, that is $a$ and $b$ do not share a digit equal to $1$, or equivalently there is no carry when calculating the sum $a+b$ in base $2$.

\begin{lemma}\label{lem:quot-Jdr-char2}
    If $\op{char}(\kk)=2$, $r=1 + 2^{r_1} + \cdots + 2^{r_k}$ ($1\leq r_1<r_2<\cdots<r_k$) is the base $2$ expansion of $r$,~then
    \[
\left[\frac{J(d-1,r-1)}{J(d,r)}\right] = \begin{cases}
    \prod_{j=1}^k F^{2^{r_j}}(h_1) & \text{if }(r,d-r)\in NC \\
    0 & \text{otherwise}.
    \end{cases}    
\]
\end{lemma}

\begin{proof}
 We argue by induction on $r$. If $r=1$ (that is, $k=0$), then $(1,d-1)\in NC$ if and only if $d$ is odd, and using \eqref{eq:quot-Jdr-r=1} and the fact that $J(d-1,0)$, $J(d,1)$ only depend on the residue class of $d$ modulo~$2$, we get
 \[\left[\frac{J(d-1,0)}{J(d,1)}\right] = \begin{cases}
     1=[S/\mf{m}] & \text{if }(1,d-1) \in NC, \\
     0 & \text{otherwise.}
 \end{cases} \]
For the induction step, we assume $k>0$ and let $q'=2^{r_k}$, $q=2q'$, so that $q'\leq r<q$. 

If $(r,d-r)\in NC$ then each of the terms $1,2^{r_1},\cdots,2^{r_k}$ appears in the $2$-adic expansion of $d$. In particular $d$ is odd, and after possibly replacing $d$ with its residue class modulo $q$ we may assume that $r\leq d < q$. We can then apply \eqref{eq:quot-consec-Jdr} to obtain
\begin{equation}\label{eq:quot-consec-Jdr-char2}
\left[\frac{J(d-1,r-1)}{J(d,r)}\right] = F^{q'}(h_1)\cdot \left[\frac{J(d'-1,r'-1)}{J(d',r')}\right] + (z_1z_2)^{q'-1}\cdot \left[\frac{J(d''-1,r''-1)}{J(d'',r'')}\right]^{\vee},
\end{equation}
where $d'=d-q'$, $r'=r-q'$, $d''=q-1-d$ and $r''=q-r$. Since $d''$ is even and $r''$ is odd, it follows that $(r'',d''-r'')\notin NC$, hence by induction we have $J(d''-1,r''-1)/J(d'',r'')=0$. Moreover $(r,d-r)\in NC$ implies $(r',d'-r')\in NC$, hence by induction we get
\[\left[\frac{J(d-1,r-1)}{J(d,r)}\right] = F^{q'}(h_1)\cdot \left[\frac{J(d'-1,r'-1)}{J(d',r')}\right] = F^{2^{r_k}}(h_1) \cdot \prod_{j=1}^{k-1}F^{2^{r_j}}(h_1) = \prod_{j=1}^{k}F^{2^{r_j}}(h_1), \]
verifying the inductive step in the case $(r,d-r)\in NC$.

Suppose now that $(r,d-r)\notin NC$, and working modulo $q$ assume that $r\leq d\leq r+q-1$. As in the proof of Corollary~\ref{cor:quot-Jdrs}, if $d\geq q$ then it follows from Theorem~\ref{thm:recursions-Idr-Jdr}(4) that $J(d-1,r-1)/J(d,r)=0$. Otherwise, we have $q'\leq r\leq d<q$ and we can apply \eqref{eq:quot-consec-Jdr-char2} again. Since $q'=2^{r_k}$ appears as a term in the $2$-adic expansion of both $r$ and $d$, the hypothesis $(r,d-r)\notin NC$ forces $(r',d'-r')\notin NC$. By induction we get $J(d'-1,r'-1)/J(d',r')=0$. Moreover, since $d''<r''$ we have $(r'',d''-r'')\not\in NC$, hence $J(d''-1,r''-1)/J(d'',r'')=0$. Using \eqref{eq:quot-consec-Jdr-char2}, we conclude that $J(d-1,r-1)/J(d,r)=0$, completing the inductive step.
\end{proof}

\section{The graded Han--Monsky representation ring}
\label{sec:Han-Monsky}

We recall the discussion of the graded Han--Monsky representation ring from the introduction. By the graded Nakayama's lemma one gets for $a\leq b$ a collection of positive integers $c_j=c_j(a,b)$, $0\leq j<a$, such that
\begin{equation}\label{eq:dela-delb-general} \delta_a\delta_b=\sum_{j=0}^{a-1}\delta_{c_j}(-j).
\end{equation}
Although important progress (ignoring the grading) has been obtained in \cite{han-monsky} and subsequent work, a complete understanding of the product formula \eqref{eq:dela-delb-general} remains unknown. The following consequence of \cite{han-monsky}*{Theorem~3.6} will be important for us.
\begin{lemma}\label{lem:del-prod-div-by-p}
    If $p|ab$ for some $1\leq a\leq b$ then $p|c_j(a,b)$ for all $j=0,\cdots,a-1$.
\end{lemma}

The connection to the results from Sections~\ref{sec:pparts-P1},~\ref{sec:coh-Fdr-recursive} arises using the exact sequence
\begin{equation}\label{eq:ses-Fdre} 0 \lra \mc{F}^d_{r}(e) \overset{\alpha}{\lra} \mc{F}^d_{r+1}(e) \oplus \mc{F}^{d-1}_{r-1}(e+1) \overset{\beta}{\lra} \mc{F}^{d-1}_{r}(e+1) \lra 0.
\end{equation}
The multiplication in the Han--Monsky ring can be rephrased in terms of the (failure of the) exactness of \eqref{eq:ses-Fdre} after taking global sections. More precisely, we have the following.

\begin{theorem}\label{thm:pparts-vs-Monsky}
    Consider non-negative integers $a,b,r,j$, and set $d=a+b-1-j$, $e=j-1$. We have
    \[\delta_r(-j)\text{ appears as a summand in }\delta_a\delta_b \Longleftrightarrow \left(\frac{J(d-1,r-1)}{J(d,r)}\right)_{a-1+r-d,b-1+r-d} \neq 0.\]
    The above is further equivalent to the assertion that $H^0(\PP^1,\beta)$ fails to be surjective in bi-degree $(a-1,b-1)$.    
\end{theorem}

\begin{proof}
    Let $A=\kk[T_1,T_2]/\langle T_1^{a},T_2^{b}\rangle$ viewed as a $\kk[T]$-module via $T\mapsto T_1+T_2$, so its isomorphism class is $\delta_{a}\delta_{b}$. If $d+e=a+b-2$, then we can identify
    \[ (D^d U \oo \Sym^e U)_{(a-1,b-1)} = A_e,\]
    and therefore (see also \eqref{eq:multigraded-comps-coh})
    \[ H^0(\PP^1,\mc{F}^d_r(e))_{(a-1,b-1)} = (0:T^r)_e.\]
    Taking global sections in \eqref{eq:ses-Fdre} and restricting to degree $(a-1,b-1)$ we get a left exact sequence
    \[0 \lra (0:T^{r})_e \overset{[\iota\ \cdot T]}{\lra} (0:T^{r+1})_e \oplus (0:T^{r-1})_{e+1} \overset{\begin{bmatrix} \cdot T \\ -\iota\end{bmatrix}}{\lra} (0:T^{r})_{e+1} \]
    where $\iota$ denotes the natural inclusion, and $\cdot T$ is multiplication by $T=T_1+T_2$. Each summand $\delta_c(-j)$ of $\delta_{a}\delta_{b}$ will therefore have a $1$-dimensional contribution to $\left(\coker H^0(\PP^1,\beta)\right)_{(a-1,b-1)}$ precisely when its cyclic generator $m$ is in degree $e+1$ (i.e. $j=e+1$) and $T^{r}m=0$ but $T^{r-1}m\neq 0$, that is, when $c=r$.

    For the relation to the ideals $J(d,r)$, we apply $H^0_*(\PP^1,-)$, take characters and compute the $T$-equivariant Euler characteristic to be
    \[ [F^d_r]-[F^d_{r+1}] - [F^{d-1}_{r-1}] + [F^{d-1}_r] = (z_1z_2)^{d-r}\cdot\left(-[J(d,r)] + [J(d-1,r-1)]\right) = (z_1z_2)^{d-r}\cdot\left[\frac{J(d-1,r-1)}{J(d,r)}\right]\]
    The monomial $z_1^{a-1}z_2^{b-1}$ appears in the above character if and only if the quotient $J(d-1,r-1)/J(d,r)$ is non-zero in bi-degree $(a-1+r-d,b-1+r-d)$, concluding our proof.
\end{proof}

Based on Theorem~\ref{thm:pparts-vs-Monsky}, we can derive some important consequences that will be crucial to our cohomology computations in the next section.

\begin{lemma}\label{lem:dela-delb-ineq}
    If $p\nmid c$ and $\delta_c(-j)$ is a summand of $\delta_a\delta_b$ then 
    \begin{equation}\label{eq:c+2j<a+b} 
     c+2j = a+b-1.
    \end{equation}
\end{lemma}

\begin{proof}
    If we let $d=a+b-1-j$, and $r=c$, it follows from Theorem~\ref{thm:pparts-vs-Monsky} that
    \[\left(\frac{J(d-1,c-1)}{J(d,c)}\right)_{a-1+c-d,b-1+c-d} \neq 0 \]
    Since $p\nmid r$, Corollary~\ref{cor:quot-Jdrs} implies that this can only happen when
    \[(a-1+c-d)+(b-1+c-d)=c-1,\]
    or equivalently $c+2j=a+b-1$, as desired.
\end{proof}

An easy induction now yields the following important generalization.

\begin{proposition}\label{prop:delcj-in-product}
    If $p\nmid c$ and $\delta_c(-j)$ is a summand of $\delta_{a_1}\cdots\delta_{a_n}$ then
    \[ c+2j = a_1+\cdots+a_n - (n-1).\]
\end{proposition}

\begin{proof}
    We prove the statement by induction on $n$, noting that $n=2$ is covered in Lemma~\ref{lem:dela-delb-ineq}. Suppose now that $n>2$ and that $\delta_c(-j)$ is a summand of $\delta_{a_1}\cdots\delta_{a_n}$. It follows that $\delta_{a_1}\cdots\delta_{a_{n-1}}$ contains a summand $\delta_b(-i)$ such that $\delta_c(-(j-i))$ is a summand of $\delta_b\cdot\delta_{a_n}$.

    By Lemma~\ref{lem:del-prod-div-by-p}, the hypothesis $p\nmid c$ implies that $p\nmid b$. We can therefore apply the induction hypothesis to conclude that
    \[ b+2i = a_1+\cdots+a_{n-1}-(n-2).\]
    Applying Lemma~\ref{lem:dela-delb-ineq} to the product $\delta_b\cdot\delta_{a_n}$, we conclude that
    \[ c+2(j-i) = b+a_n-1.\]
    Adding the two displayed equalities above yields the induction step and concludes our proof.
\end{proof}

While Proposition~\ref{prop:delcj-in-product} uniquely identifies the degree shift of a summand $\delta_c$, it is not a priori clear for which values of $c$ with $p\nmid c$ is there a summand $\delta_c$ of $\delta_{a_1}\cdots\delta_{a_n}$. We show next that in characteristic $p=2$ this summand is uniquely determined. As in \cite{gao-raicu}*{Introduction}, \cite{GRV}*{Section~5.3}, we write $a\oplus b$ for the \defi{Nim-sum} of two non-negative integers $a,b$.

\begin{proposition}\label{prop:delab-odd-Nim}
    Suppose that $\op{char}(\kk)=2$. For $a,b,c\geq 0$, we have that $\delta_{2c+1}(-j)$ appears as a summand of $\delta_{2a+1}\delta_{2b+1}$ if and only if
    \[ c = a\oplus b \quad\text{ and }\quad j = a+b-c.\]
\end{proposition}

\begin{proof}
 We let $r=2c+1$ and consider positive integers $1\leq r_1<\cdots<r_k$ such that
 \[ r = 1 + 2^{r_1} + \cdots + 2^{r_k},\]
 or equivalently, $c=2^{r_1-1} + \cdots + 2^{r_k-1}$.
 
 Suppose first that $c=a\oplus b$ and $j=a+b-c$. The condition $c=a\oplus b$ implies that there exists a non-negative integer $m$ and a partition $\{1,\cdots,k\} = A \sqcup B$ such that
 \begin{equation}\label{eq:ab=m+stuff} 
 a = m + \sum_{j\in A} 2^{r_j-1}\quad\text{ and }\quad b = m + \sum_{j\in B} 2^{r_j-1},
 \end{equation}
 where none of $2^{r_j-1}$ is a term in the $2$-adic expansion of $m$. We can then write $j=a+b-c = 2m$. If we define $d = (2a+1)+(2b+1)-1-j = r + 2m$, then we have that $(r,d-r)\in NC$, hence Lemma~\ref{lem:quot-Jdr-char2} implies
 \[\left[\frac{J(d-1,r-1)}{J(d,r)}\right] = \prod_{j=1}^k F^{2^{r_j}}(h_1) = (z_1^{2^{r_1}}+z_2^{2^{r_1}})\cdots (z_1^{2^{r_k}}+z_2^{2^{r_k}}). \]
Using the fact that
\[ (2a+1)-1+r-d = \sum_{j\in A}2^{r_j}\quad\text{ and }\quad (2b+1)-1+r-d = \sum_{j\in B}2^{r_j},\]
we can apply Theorem~\ref{thm:pparts-vs-Monsky} to conclude that $\delta_{2c+1}(-j)$ appears as a summand in $\delta_{2a+1}\delta_{2b+1}$, as desired.

Conversely, if we assume that $\delta_{2c+1}(-j)$ appears as a summand in $\delta_{2a+1}\delta_{2b+1}$, then it follows from Lemma~\ref{lem:dela-delb-ineq} that $j=a+b-c$. We let $d = (2a+1)+(2b+1)-1-j = a+b+c+1$, and apply Lemma~\ref{lem:quot-Jdr-char2} to deduce that $(r,d-r)=(2c+1,a+b-c)\in NC$. Moreover, using the fact that
\[ (2a+1)-1+r-d = c + a - b\quad\text{ and }\quad (2b+1)-1+r-d = c-a+b,\]
we conclude by combining Lemma~\ref{lem:quot-Jdr-char2} with Theorem~\ref{thm:pparts-vs-Monsky} that
\[ z_1^{c+a-b}\cdot z_2^{c+b-a} \text{ appears as a term in the expansion of }(z_1^{2^{r_1}}+z_2^{2^{r_1}})\cdots (z_1^{2^{r_k}}+z_2^{2^{r_k}}),\]
and in particular $(c+a-b,c+b-a)\in NC$. Since the sum $(c+a-b)+(c+b-a)=2c$, it follows that there exists a partition $\{1,\cdots,k\} = A \sqcup B$ such that
 \[ c+a-b = \sum_{j\in A} 2^{r_j}\quad\text{ and }\quad c+b-a = \sum_{j\in B} 2^{r_j}.\]
 If we define $m=(a+b-c)/2$ then \eqref{eq:ab=m+stuff} holds, and we have $m\geq 0$ because $(2c+1,a+b-c)\in NC$. Moreover,
 \[ {a+b+c+1 \choose a+b-c} = {2m+2c+1 \choose 2m} \overset{(\op{mod}\ 2)}\equiv {m+c \choose c},\]
 and therefore $(m,c)\in NC$. This implies that none of $2^{r_j-1}$ (which are the terms in the $2$-adic expansion of $c$) is a term in the $2$-adic expansion of $m$. Using \eqref{eq:ab=m+stuff} again yields $a\oplus b = c$ and concludes the proof.
\end{proof}

By induction, we get the following important consequence.

\begin{corollary}\label{cor:prod-dela-odd-Nim}
    Suppose that $\op{char}(\kk)=2$. If $a_1,\cdots,a_n,c\geq 0$, we have that $\delta_{2c+1}(-j)$ appears as a summand of $\delta_{2a_1+1}\delta_{2a_2+1}\cdots\delta_{2a_n+1}$ if and only if
    \[ c = a_1\oplus a_2\oplus\cdots\oplus a_n \quad\text{ and }\quad j = a_1+\cdots+a_n-c.\]
\end{corollary}

\section{Recursive description of cohomology and consequences}\label{sec:recursive-cohomology}

We let $V=\kk^n$, let $\PP^{n-1}=\bb{P}V$, and consider the sheaves $\mc{F}^d_r$ defined (in analogy with the $\bb{P}^1$ case) by
\[ 0 \lra \mc{F}^d_r \lra D^dV \oo \mc{O}_{\PP^{n-1}} \overset{\Delta}{\lra} D^{d-r}V \oo \mc{O}_{\PP^{n-1}}(r) \lra 0.\]
By the long exact sequence in cohomology, it follows that if $e\geq -1$ (in fact $e>-n$) then $\mc{F}^d_r(e)$ can only have non-zero cohomology in degrees $i=0,1$.

With $S$, $M$ as in Section~\ref{sec:prelim}, we can view $\Delta$ as the sheafification of the map of graded $S$-modules
\[ M_{d,\bullet} = D^dV \oo S \overset{\cdot\omega^r}{\lra} M_{d-r,\bullet+r} = D^{d-r}V \oo S,  \]
so we will usually identify $\Delta = \cdot\omega^r$. We get for $e\geq -1$ an exact sequence
\[ 0 \lra H^0\left(\bb{P}^{n-1},\mc{F}^d_r(e)\right) \lra M_{d,e} \overset{\cdot\omega^r}{\lra} M_{d-r,e+r} \lra H^1\left(\bb{P}^{n-1},\mc{F}^d_r(e)\right) \lra 0\]
Moreover, we have a commutative diagram
\[
\xymatrix{
& & 0 \ar[d] & & \\
& & D^d\mc{R} \ar[d] &  & \\
0 \ar[r] & \mc{F}^d_r \ar[r] \ar[d] & D^dV \oo \mc{O}_{\PP^{n-1}} \ar[r]^{\cdot\omega^{r}} \ar[d]^{\cdot\omega} & D^{d-r}V \oo \mc{O}_{\PP^{n-1}}(r) \ar[r] \ar@{=}[d] & 0 \\
0 \ar[r] & \mc{F}^{d-1}_{r-1}(1) \ar[r] & D^{d-1}V \oo \mc{O}_{\PP^{n-1}}(1) \ar[r]^{\cdot\omega^{r-1}} \ar[d] & D^{d-r}V \oo \mc{O}_{\PP^{n-1}}(r) \ar[r] & 0 \\
& & 0 & & \\
}
\]
where each displayed short exact sequence is exact. We get an induced short exact sequence
\[ 0 \lra D^d\mc{R} \lra \mc{F}^d_r \lra \mc{F}^{d-1}_{r-1}(1) \lra 0.\]
The following result establishes a sufficient condition for the vanishing of the connecting homomorphism in the cohomology long exact sequence associated to suitable twists of the above sequence.

\begin{theorem}\label{thm:lefschetz-Fdr}
    Consider the short exact sequence
    \[ 0 \lra D^d\mc{R}(e) \lra \mc{F}^d_r(e) \lra \mc{F}^{d-1}_{r-1}(e+1) \lra 0.\]
    We have for $2\leq r\leq p$, $e\geq d-1$ and all $i$ a short exact sequence
    \[ 0 \lra H^i\left(\bb{P}^{n-1},D^d\mc{R}(e)\right) \lra H^i\left(\bb{P}^{n-1},\mc{F}^d_r(e)\right) \lra H^i\left(\bb{P}^{n-1},\mc{F}^{d-1}_{r-1}(e+1)\right) \lra 0.\]
\end{theorem}

\begin{proof}
    Since $e\geq d-1\geq -1$, we noted above that cohomology can only be non-zero for $i=0,1$. To establish the result, it is then enough to check exactness when $i=0$. We recall that there is a natural $\bb{Z}^n$-grading on cohomology, and we establish the exactness on multigraded component. 

    We fix a multidegree $\ul{a}=(a_1,\cdots,a_n)\in\bb{Z}^n_{\geq 0}$ with $d+e=a_1+\cdots+a_n$, and the corresponding Artinian algebra
    \[ A = \kk[T_1,\cdots,T_n]/\langle T_1^{1+a_1},\cdots,T_n^{1+a_n}\rangle.\]
    Recalling that $T=T_1+\cdots+T_n$, and its identification with $\omega$ from Section~\ref{sec:prelim}, we observe that the desired surjectivity in multidegree $\ul{a}$ is equivalent to proving that inside the algebra $A$, the natural map
    \[ (0:T^r)_e \overset{\cdot T}{\lra} (0:T^{r-1})_{e+1}\]
    is surjective. If we suppose by contradiction that surjectivity fails, then there exists a summand $\delta_k(-e-1)$ of $\delta_{a_1+1}\cdots \delta_{a_n+1}$ with $1\leq k\leq r-1$. Note that the hypothesis $r\leq p$ implies that $p\nmid k$. By Proposition~\ref{prop:delcj-in-product}, it follows that
    \[ k+2(e+1) = (a_1+1) + \cdots + (a_n+1) - (n-1) = d + e + 1,\]
    or equivalently, that
    \[ k = d-e-1.\]
    The hypothesis $e\geq d-1$ implies $k\leq 0$, which is a contradiction and concludes our proof.
\end{proof}

Our next result shows that the cohomology of $\mc{F}^d_p$ is related, up to Frobenius twists, to that of $D^d\mc{R}$, which is the key observation leading up to our recursive description of cohomology. Essentially, this is based on the simple identity
\[ \omega^p = x_1^py_1^p+\cdots+x_n^py_n^p,\]
which holds under the assumption that $\op{char}(\kk)=p$. What it says is that the $S$-module map
\begin{equation}\label{eq:mult-omegap-onS} 
D^d V \oo S \overset{\cdot\omega^p}{\lra}  D^{d-p}V \oo S
\end{equation}
is defined over the polynomial subring $S^{(p)}=\Sym(F^pV) = \kk[x_1^p,\cdots,x_n^p]$.

\begin{theorem}\label{thm:frobenius-Fdp}
    We let $\phi:\bb{P}V\lra \bb{P}(F^pV)$ denote the Frobenius morphism, induced by the inclusion of homogeneous coordinate rings $\Sym(F^pV)\subset\Sym(V)$. We have
    \[ \phi_*(\mc{F}^d_p(e)) = \bigoplus_{\substack{0\leq a\leq d/p \\ b\leq e/p}} T_p\Sym^{d-pa}V \oo T_p\Sym^{e-pb}V \oo D^a\tilde{\mc{R}}(b),\]
    where $\tilde{\mc{R}}$ denotes the tautological subsheaf on $\bb{P}(F^pV)$.
\end{theorem}

\begin{proof} We have decompositions
    \[ D^d V = \bigoplus_{0\leq a\leq d/p} T_p\Sym^{d-pa}V \oo F^p(D^a V),\quad D^{d-p} V = \bigoplus_{1\leq a\leq d/p} T_p\Sym^{d-pa}V \oo F^p(D^{a-1} V)\]
which is equivariant for the action of the maximal torus in $\GL(V)$. The summands involving $T_p\Sym^{d-pa}V$ correspond to the multigraded components of $D^d V$ and $D^{d-p}V$ in multidegrees $\ul{a}\in\bb{Z}^n_{\geq 0}$, where
\[ \ol{a}_1+\cdots+\ol{a}_n = d-pa,\]
and $0\leq\ol{a}_i <p$ denotes the remainder of the division of $a_i$ by $p$. It follows that the map
\[ D^d V \oo S^{(p)} \overset{\cdot\omega^p}{\lra}  D^{d-p}V \oo S^{(p)},\]
which defines \eqref{eq:mult-omegap-onS} after base change $-\oo_{S^{(p)}} S$, decomposes further as a direct sum of maps
\[ T_p\Sym^{d-pa}V \oo F^p(D^a V) \oo S^{(p)} \overset{\cdot\omega^p}{\lra} T_p\Sym^{d-pa}V \oo F^p(D^{a-1} V) \oo S^{(p)},\]
which can be further identified with
\[T_p\Sym^{d-pa}V \oo F^p\left(D^a V \oo S\right) \overset{\op{id}\oo F^p(\cdot\omega)}{\lra} T_p\Sym^{d-pa}V \oo F^p\left(D^{a-1} V \oo S\right). \]
Applying base change $-\oo_{S^{(p)}} S$ and sheafifying, we get that
\[ \mc{F}^d_p = \bigoplus_{a\leq d/p} T_p\Sym^{d-ap}V \oo \phi^*\left(D^a\tilde{R}\right).\]
Applying $\phi_*$ and using the projection formula, along with the decomposition
\[ \phi_*(\mc{O}_{\bb{P}V}(e)) = \bigoplus_b T_p\Sym^{e-pb}V \oo \mc{O}_{\bb{P}(F^pV)}(b),\]
we obtain the desired conclusion.
\end{proof}

We now have all the ingredients necessary to prove the main result of our paper, which gives the recursive description for the cohomology of $D^d\mc{R}(e)$.

\begin{proof}[Proof of Theorem~\ref{thm:coh-recursion}] Taking characters and combining Theorems~\ref{thm:lefschetz-Fdr} and~\ref{thm:frobenius-Fdp} it follows that for $e\geq d-1$ we have
\[
\begin{aligned}
    h^i(D^d\mc{R}(e)) - h^i(D^{d-p}\mc{R}(e+p)) &= h^i(\mc{F}^d_p(e)) - h^i(\mc{F}^{d-1}_p(e+1)) \\
    &= \sum_{a,b} (h'_{e-bp}\cdot h'_{d-ap}-h'_{e+1-bp}\cdot h'_{d-1-ap})\cdot F^p\left(h^i(D^a\mc{R}(b)\right).
\end{aligned}
\]
We can now replace $e$ by $e+jp$ and $d$ by $d-jp$ for $j\geq 0$, and take the sum over all such $j$ to get
\[
\begin{aligned}
h^i(D^d\mc{R}(e)) &= \sum_{j\geq 0} h^i(D^{d-jp}\mc{R}(e+jp)) - h^i(D^{d-(j+1)p}\mc{R}(e+(j+1)p)) \\
&= \sum_{a,b} \Phi_{d-ap,e-bp}\cdot F^p\left(h^i(D^a\mc{R}(b)\right)
\end{aligned}
\]
which proves the desired identity.
\end{proof}

We conclude this section with a proof of \cite{GRV}*{Conjecture~5.1}, which we restate as follows.

\begin{theorem}\label{thm:small-weights}
 Suppose that $tp\leq d<(t+1)p$, where $1\leq t<p$. We have for all $e\geq d-1$
 \[ h^1(D^d\mc{R}(e)) = \sum_{\substack{1\leq B\leq A\leq t \\ 0\leq J\leq A-B}} \left(h'_{e+(B-J)p}\cdot h'_{d-Ap} - h'_{e+(B-J)p+1}\cdot h'_{d-Ap-1}\right) \cdot F^p(s_{(A-B,J)}).\]
\end{theorem}

\begin{proof} It follows from Theorem~\ref{thm:coh-recursion} that
\[ h^1(D^d\mc{R}(e)) = \sum_{a=0}^t\sum_{b\in\bb{Z}}\sum_{j\geq 0}\left(h'_{e-bp+jp}\cdot h'_{d-ap-jp} - h'_{e-bp+1+jp}\cdot h'_{d-ap-1-jp}\right)\cdot F^p\left(h^1(D^a\mc{R}(b)\right).\]
Since $a<p$, we have as in characteristic zero
\[ h^1(D^a\mc{R}(b)) = \begin{cases}
    s_{(a-1,b+1)} & \text{if }-1\leq b\leq a-2, \\
    0 & \text{otherwise}.
\end{cases}\]
We can therefore restrict the above sum to $a\geq 1$, $-1\leq b\leq a-2$, and also restrict to $a+j\leq t$ (since otherwise $h'_{d-ap-jp}=h'_{d-ap-1-jp}=0$). If we set $A=a+j$, $B=j+1$, $J=b+1$, then $A-B = a-1 \geq b+1 = J \geq 0$ and $t\geq A\geq B\geq 1$. Moreover, we have
\[e+(B-J)p = e - bp + jp,\quad\text{and}\quad d-Ap = d-ap-jp.\]
It follows that the conclusion of Theorem~\ref{thm:coh-recursion} agrees in this case with our desired description of $h^1(D^d\mc{R}(e))$.
\end{proof}

\section{Cohomology in characteristic $2$ and Nim}\label{sec:char-2-Nim}

Throughout this section we assume that $\op{char}(\kk)=2$, and $V=\kk^n$. Our goal is to understand the generating function
\[ \bG(u,v) = \sum_{d,e} h^1(D^d\mc{R}(e))\cdot u^d\cdot v^{d+e},\]
noting that $\bG=\bG(u,v) \in \Lambda[[u,v]]$, where $\Lambda$ is as in \eqref{eq:def-Lambda}. We define
\[\bh'(t) = \sum_{d\geq 0}h'_d \cdot t^d = \prod_{i=1}^n(1+tz_i)\quad\text{and}\quad\bN(t) = \sum_{m\geq 0} \mc{N}_m\cdot t^m,\]
where $\mc{N}_m$ denotes the Nim polynomial \eqref{eq:def-Nim-pol}. The Frobenius action on $\Lambda$ (given by $F^2(z_i)=z_i^2$) extends to $\Lambda[[t]]$ and $\Lambda[[u,v]]$ by sending 
\[F^2(t)=t^2,\quad F^2(u)=u^2,\quad F^2(v)=v^2.\]

\begin{theorem}\label{thm:functional-eqn-G}
    With the notation above, $G$ satisfies the functional equation
    \[ (1+u)\cdot\bG = \bh'(uv)\cdot \bh'(v)\cdot F^2(\bG) + u\cdot F^2\left(\bN(uv^2)\right).\]
\end{theorem}

\begin{proof}
 We consider $d\geq 0$, $e\geq -d$, and as in the proof of Theorem~\ref{thm:coh-recursion}, we fix a multidegree $\ul{a}=(a_1,\cdots,a_n)\in\bb{Z}^n_{\geq 0}$. The sequence
\[ 0 \lra H^1\left(\bb{P}^{n-1},D^d\mc{R}(e)\right)_{\ul{a}} \lra H^1\left(\bb{P}^{n-1},\mc{F}^d_2(e)\right)_{\ul{a}} \lra H^1\left(\bb{P}^{n-1},D^{d-1}\mc{R}(e+1)\right)_{\ul{a}} \lra 0\]
fails to be exact if and only if $\delta_1(-(e+1))$ appears as a summand in $\delta_{a_1+1}\cdots\delta_{a_n+1}$, which by Lemma~\ref{lem:del-prod-div-by-p} and Corollary~\ref{cor:prod-dela-odd-Nim} is equivalent to the existence of integers $i_1,\cdots,i_n\in\bb{Z}_{\geq 0}$ with
\[ a_1=2i_1,\cdots,a_n=2i_n,\ i_1\oplus\cdots\oplus i_n=0,\ e+1 = i_1+\cdots+i_n.\]
This is further equivalent to $e+1=(d+e)/2$ and $z_1^{i_1}\cdots z_n^{i_n}$ being a monomial in $\mc{N}_{(d-1)/2}$. Taking characters and summing over all $d,e$ and all multidegrees $\ul{a}$ we conclude
\[ \sum_{d,e} \left(h^1(D^d\mc{R}(e))+h^1(D^{d-1}\mc{R}(e+1))\right)\cdot u^d\cdot v^{d+e} = \left(\sum_{d,e} h^1(\mc{F}^d_2(e))\cdot u^d\cdot v^{d+e}\right) + u\cdot F^2\left(\bN(uv^2)\right).\]
The left side is $(1+u)\cdot\bG$, so we need to understand the first term on the right side. We can apply Theorem~\ref{thm:frobenius-Fdp} to conclude
\[
\begin{aligned}
\sum_{d,e} h^1(\mc{F}^d_2(e))\cdot u^d\cdot v^{d+e} &= \sum_{d,e,a,b} h'_{d-2a}\cdot h'_{e-2b}\cdot F^2\left(h^1(D^a\mc{R}(b))\right)\cdot u^d\cdot v^{d+e} \\
&=\sum_{d,e,a,b} \left(h'_{d-2a}\cdot (uv)^{d-2a}\right) \cdot \left(h'_{e-2b}\cdot v^{2b}\right)\cdot F^2\left(h^1(D^a\mc{R}(b))\cdot u^a\cdot v^{a+b}\right) \\
&= \bh'(uv) \cdot \bh'(v)\cdot F^2(\bG),
\end{aligned}
\]
concluding our proof.
\end{proof}

To translate Theorem~\ref{thm:functional-eqn-G} into a more concrete formula for the cohomology of $D^d\mc{R}(e)$, we introduce more notation: for $q=2^k\geq 2$ we write
\[\bh^{(q)}(t) = \sum_{d\geq 0} h^{(q)}_d \cdot t^d,\]
so that $\bh' = \bh^{(2)}$. The following corrects \cite{GRV}*{Conjecture~5.2}.

\begin{corollary}\label{cor:coh-Ddr-using-Nim}
    If $\op{char}(\kk)=2$ and $e\geq d-1$ then
 \[ \left[H^1\left(\PP^{n-1},D^d\mc{R}(e)\right)\right] = \sum_{(q,m,j)\in\Lambda_d} F^{2q}(\mc{N}_m)\cdot s^{(q)}_{(e-(2m-2j-1)q,d-(2m+2j+1)q)},\]
 where $\Lambda_d = \left\{(q,m,j) | q=2^r\text{ for some }r\geq 1,\ m,j\geq 0,\text{ and }(2m+2j+1)q\leq d\right\}$.
\end{corollary}

\begin{proof}
    We first note that if $q=2q'$ then
    \[ \bh^{(q)} = \bh'\cdot F^2(\bh')\cdot F^4(\bh')\cdots F^{q'}(\bh').\]
    Similarly, we have
    \[ 1 + u + \cdots + u^{q-1} = (1+u)\cdot F^2(1+u)\cdot F^4(1+u)\cdots F^{q'}(1+u).\]
    It follows that if we define $\bg$ via
    \[ \bg(u,v) = \frac{\bh'(uv)\cdot \bh'(v)}{1+u}\]
    then
    \[ \bg\cdot F^2(\bg)\cdot F^4(\bg)\cdots F^{q'}(\bg) = \frac{\bh^{(q)}(uv)\cdot \bh^{(q)}(v)}{1 + u + \cdots + u^{q-1}}. \]
    We can now rewrite and iterate the formula in Theorem~\ref{thm:functional-eqn-G} to get
    \[ 
    \begin{aligned}
      G &= \frac{u}{1+u}\cdot F^2(\bN(uv^2)) + \bg\cdot F^2(G) \\
      &= \frac{u}{1+u}\cdot F^2(\bN(uv^2)) + \bg\cdot\frac{u^2}{1+u^2}\cdot F^4(\bN(uv^2))+\bg\cdot F^2(\bg)\cdot\frac{u^4}{1+u^4}\cdot F^8(\bN(uv^2))+\cdots
    \end{aligned}
    \]
    For $q=2q'$ powers of $2$, the general term in the above sum is
    \[
    \begin{aligned}
        \bg\cdot & F^2(\bg)\cdots F^{q'}(\bg)\cdot\frac{u^q}{1+u^q}\cdot F^{2q}(\bN(uv^2)) = \frac{\bh^{(q)}(uv)\cdot \bh^{(q)}(v)\cdot u^q}{1 + u + \cdots + u^{2q-1}} \cdot \sum_{m\geq 0} F^{2q}(\mc{N}_m)\cdot u^{2qm}\cdot v^{4qm} \\
        &= \bh^{(q)}(uv)\cdot \bh^{(q)}(v)\cdot(1-u)\cdot\left(\sum_{j\geq 0} u^{(2j+1)q}\right)\cdot \left(\sum_{m\geq 0} F^{2q}(\mc{N}_m)\cdot u^{2qm}\cdot v^{4qm}\right)
    \end{aligned}
    \]
    If we write $c_{a,b}$ for the coefficient of $u^av^b$ in $\bh^{(q)}(uv)\cdot \bh^{(q)}(v)\cdot(1-u)$ then we have for $b\geq 2a-1$ that
    \[ c_{a,b} = h^{(q)}_{b-a}\cdot h^{(q)}_a-h^{(q)}_{b-a+1}\cdot h^{(q)}_{a-1} = s^{(q)}_{(b-a,a)}.\]
    It follows that we can rewrite
    \[ \bG = \frac{u}{1+u}\cdot F^2(\bN(uv^2)) + \sum_{\substack{q=2^r\geq 2 \\ a,b,m,j\geq 0}} F^{2q}(\mc{N}_m)\cdot c_{a,b} \cdot u^{a+(2m+2j+1)q}\cdot v^{4mq+b}.\]
    To compute the coefficient of $u^d v^{d+e}$ for $e\geq d-1$ we set $a=d-(2m+2j+1)q$, $b=d+e-4mq$ and note that the inequality $e\geq d-1$ forces $b\geq 2a-1$. Moreover, for $e\geq d-1$, the coefficient of $u^dv^{d+e}$ in $\frac{u}{1+u}\cdot F^2(\bN(uv^2))$ is $0$. Using the fact that $b-a=e-(2m-2j+1)q$ it follows that the coefficient of $u^d v^{d+e}$ in $\bG$ is
    \[h^1(D^d\mc{R}(e)) = \sum_{\substack{q=2^r\geq 2 \\ m,j\geq 0}} F^{2q}(\mc{N}_m)\cdot s^{(q)}_{(e-(2m-2j-1)q,d-(2m+2j+1)q)}. \]
    Restricting further to parameters $m,j,q$ for which $d\geq (2m+2j+1)q$ (since otherwise the truncated Schur polynomial factor vanishes), we obtain the desired conclusion.
\end{proof}

\section{Weak Lefschetz for Artinian monomial complete intersections}\label{sec:Lefschetz}

In this section we consider the Artinian monomial complete intersections
    \[M_{\ul{a}} = \kk[T_1,\cdots,T_n]/\langle T_1^{1+a_1},\cdots, T_n^{1+a_n}\rangle\]
and study the Weak Lefschetz Property (WLP) as a function of $a_1,\cdots,a_n$ and the characteristic of $\kk$. It is well-known that WLP holds for all $\ul{a}$ if $\op{char}(\kk)=0$ \cites{stanley,watanabe} (in fact, the Strong Lefschetz Property also holds, since $M_{\ul{a}}$ is the cohomology ring of a product of projective spaces), but the problem is more subtle in positive characteristic, and it is closely related to our earlier computation of cohomology, as explained next.

If we write $s=a_1+\cdots+a_n$ for the socle degree of $M_{\ul{a}}$, then WLP fails for $M_{\ul{a}}$ if and only if
    \[ \left(\frac{M_{\ul{a}}}{T\cdot M_{\ul{a}}}\right)_{e} \neq 0 \quad\text{ for some }e \geq \left\lfloor\frac{s}{2}\right\rfloor+1,\text{ if and only if}\quad\left(\frac{M_{\ul{a}}}{T\cdot M_{\ul{a}}}\right)_{e} \neq 0 \quad\text{ for }e = \left\lfloor\frac{s}{2}\right\rfloor+1,\]
    where $T=T_1+\cdots+T_n$. Using \eqref{eq:multigraded-comps-coh}, it follows that $M_{\ul{a}}$ fails WLP if and only if 
    \begin{equation}\label{eq:coh-characterization-WLP}
    H^1(\PP^{n-1},D^d\mc{R}(e))_{\ul{a}} \neq 0\text{ for some }d,e\text{ with }d+e=s,\ e\geq d-1.
    \end{equation}
It follows that failure of WLP is characterized by the \defi{support sets}
\begin{equation}\label{eq:def-supp-d-e} \op{supp}(d,e) = \left\{ \ul{a}\in\bb{Z}^n : H^1\left(\PP^{n-1},D^d\mc{R}(e)\right)_{\ul{a}} \neq 0\right\}
\end{equation}
which in addition satisfy the inclusion relations
\begin{equation}\label{eq:incl-supp-de}
    \op{supp}(d,e) \supseteq \op{supp}(d-f,e+f)\quad\text{ for }f\geq 0,\ e\geq d-1.
\end{equation}
    
The characterization \eqref{eq:coh-characterization-WLP} immediately implies that WLP holds for all $\ul{a}$ when $n=2$ \cites{HMNW,MZ}: indeed, we have $D^d\mc{R}(e) = \mc{O}_{\PP^1}(e-d)$, and the condition $e\geq d-1$ forces $H^1(\PP^1,\mc{O}_{\PP^1}(e-d))=0$. We will therefore focus on $n\geq 3$. In Section~\ref{subsec:WLP-from-socle} we characterize the socle degrees for which, in a fixed characteristic~$p$, all $M_{\ul{a}}$ satisfy WLP, while in Section~\ref{subsec:WLP-char2} we characterize the tuples $\ul{a}$ for which the algebras $M_{\ul{a}}$ satisfy WLP when $\op{char}(\kk)=2$.
       
\subsection{Socle degrees that guarantee the Weak Lefschetz Property}\label{subsec:WLP-from-socle} Throughout this section $\kk$ denotes a field of characteristic $p>0$, and $q$ denotes a power of $p$. Our goal is to prove Theorem~\ref{thm:WLP-from-socle}, describing the positive integers $s$ such that all monomial complete intersections $M_{\ul{a}}$ with socle degree $s$ satisfy WLP. 


\begin{proof}[Proof of Theorem~\ref{thm:WLP-from-socle}]
 We fix a socle degree $s$, and set $d=e=s/2$ if $s$ is even, and $d=(s+1)/2$, $e=(s-1)/2$ if $s$ is odd. The assertion that WLP holds for all algebras $M_{\ul{a}}$ with $s=a_1+\cdots+a_n$ is then equivalent to the cohomology vanishing
 \begin{equation}\label{eq:H1vanishing-smalld} H^1\left(\PP^{n-1},D^d\mc{R}(e)\right) = 0.\end{equation}
 Notice that the condition $s\leq 2p-2$ is equivalent to $d<p$, in which case \eqref{eq:H1vanishing-smalld} holds by \eqref{eq:h01-for-small-d} since $e\geq d-1$. We may then assume that $s\geq 2p-1$, or $d\geq p$, and find unique integers $t,q=p^k$ with
 \[ 1\leq t< p\leq q,\quad tq\leq d < (t+1)q.\]
 Using \cite{gao-raicu}*{(3.14)}, the vanishing \eqref{eq:H1vanishing-smalld} is equivalent to
 \begin{equation}\label{eq:H1van-bound-e} e \geq (t+n-2)q-n+1.\end{equation}
 
 Suppose first that $s$ is even, so $e=d\leq(t+1)q-1$, in which case \eqref{eq:H1van-bound-e} implies $n-2\geq(n-3)q\geq 2(n-3)$, which can only hold if $n=3,4$. If $n=3$ then we get $(t+1)q-2\leq e\leq (t+1)q-1$, hence
 \[ s = (2t+2)q-4 \quad \text{ or }\quad s = (2t+2)q-2.\]
 If $n=4$ then the inequalities $(t+1)q-1\geq e\geq (t+2)q-3$ hold if and only if $q=2$, which in turn forces $p=2$, $t=1$, and $e=3$, hence $s=6$. This shows that conditions (2), (3) are necessary when $s$ is even. If instead $s$ is odd, then $e=d-1\leq(t+1)q-2$. Condition \eqref{eq:H1van-bound-e} implies $n-3\geq (n-3)q$ which can only hold when $n=3$, which then forces $e=(t+1)q-2$, and $s=2e+1=(2t+2)q-3$, proving that (2), (3) are necessary also when $s$ is odd. It is now clear that in both cases (2), (3), the condition \eqref{eq:H1van-bound-e} is satisfied, which shows that conditions (2), (3) are sufficient and concludes our proof.
\end{proof}

\subsection{Weak Lefschetz Property in characteristic $2$}\label{subsec:WLP-char2} Throughout this section $\kk$ denotes a field of characteristic $p=2$, and $q$ denotes a power of $2$. Given $x\in\bb{Z}$ and $q=2^k$, $k\geq 1$, we write
\[ x = 2qm + r\text{ with }0\leq r<2q,\]
and define 
\begin{equation}\label{eq:def-theta-q}
 \theta_q(x) = \begin{cases}
     r & \text{if }r\leq q-1,\\
     2q-2-r & \text{if }q-1\leq r\leq 2q-2,\\
     -\infty & \text{if }r=2q-1.
 \end{cases}
\end{equation}

\begin{theorem}\label{thm:WLP-char-2}
    If $\op{char}(\kk)=2$ then the Weak Lefschetz Property holds for $M_{\ul{a}}$ if and only if
    \begin{equation}\label{eq:WLP-cond-char2} \sum_{i=1}^n \theta_q(a_i) \leq 2q-2 \text{ for all }q=2^k\quad\text{ with }\quad 2q>a_1\oplus a_2\oplus\cdots\oplus a_n.
    \end{equation}
\end{theorem}

Note that condition~\eqref{eq:WLP-cond-char2} is non-trivial only for finitely many values of $q$. Indeed, if $\sum_{i=1}^n a_i < q$ then
\[ \sum_{i=1}^n \theta_q(a_i) = \sum_{i=1}^n a_i \leq q-1 \leq 2q-2.\]
Even better, if we know that \eqref{eq:WLP-cond-char2} holds for some $q>\max_{i=1}^n(a_i)$ then it must also hold for all $\tilde{q}\geq q$, since the hypothesis implies $\theta_q(a_i)=a_i=\theta_{\tilde{q}}(a_i)$. It follows from Theorem~\ref{thm:WLP-char-2} that if we write $m=\max_{i=1}^n(a_i)$ then WLP holds for $M_{\ul{a}}$ if and only if
\[ \theta_q(a_i) \leq 2q-2 \quad\text{ for all }q=2^k\quad\text{ with }\quad \frac{a_1\oplus a_2\oplus\cdots\oplus a_n}{2} < q \leq 2m.\]

\begin{example}\label{ex:WLP-n=6}
    Suppose $n=6$ and $\ul{a}=(10, 3, 3, 2, 1, 1)$. We have $m=10$ and
    \[ a_1\oplus a_2\oplus\cdots\oplus a_n = 8,\]
    so it suffices to consider $q=8$ and $q=16$. We have $\theta_8(10) = 4$ and in all other cases $\theta_q(a_i)=a_i$. It follows that
    \[ \sum_{i=1}^n \theta_q(a_i) = \begin{cases}
        14 & \text{if }q = 8,\\
        20 & \text{if }q=16,
    \end{cases}\]
    and in both cases the inequality \eqref{eq:WLP-cond-char2} holds. It follows that
    \[\kk[T_1,T_2,T_3,T_4,T_5,T_6]/\langle T_1^{11}, T_2^4, T_3^4, T_4^3, T_5^2, T_6^2\rangle\]
    satisfies WLP in characteristic $2$. If we consider instead $\ul{a}=(9, 3, 2, 2, 2, 1)$ then $m=9$ and 
    \[ a_1\oplus a_2\oplus\cdots\oplus a_n = 9,\]
    so we again have to consider $q=8$ and $q=16$. We have
    \[ \sum_{i=1}^n \theta_q(a_i) = 15 > 2q-2 \quad\text{ for }q=8\]
    hence $WLP$ fails for 
    \[\kk[T_1,T_2,T_3,T_4,T_5,T_6]/\langle T_1^{10}, T_2^4, T_3^3, T_4^3, T_5^3, T_6^2\rangle.\]
\end{example}

\begin{lemma}\label{lem:thetaq-meaning}
    Consider integers $q\geq 1$ and $0\leq r\leq 2q-2$. For $\Delta\geq 0$, we can find $a,b$ satisfying
    \begin{equation}\label{eq:a-b-del-q} 
    r = a+b,\quad a-b \geq \Delta,\quad 0\leq a,b<q
    \end{equation}
    if and only if $\Delta \leq \theta_q(r)$.
\end{lemma}

\begin{proof}
    For the \emph{only if} direction, suppose that there exist $a,b\geq 0$ satisfying \eqref{eq:a-b-del-q}. If $r\leq q-1$ then 
    \[ \theta_q(r) = r = a+b \geq a-b \geq \Delta.\]
    If instead $q\leq r \leq 2q-2$ then $a\leq q-1$ and $b=r-a \geq r-(q-1)$, hence
    \[ \Delta \leq a-b \leq (q-1) - (r-(q-1)) = \theta_q(r), \]
    as desired.

    For the \emph{if} direction, suppose that $\Delta\leq\theta_q(r)$. If $r\leq q-1$ then we can take $a=r$, $b=0$, in which case we have $a-b=r=\theta_q(r)\geq\Delta$, and \eqref{eq:a-b-del-q} holds. If $q\leq r\leq 2q-2$ then $a=q-1$ and $b=r-(q-1)$ satisfy \eqref{eq:a-b-del-q}, concluding our proof.
\end{proof}

As a direct consequence of Lemma~\ref{lem:thetaq-meaning}, we get the following characterization of the terms that appear in a tensor product of truncated symmetric polynomials.

\begin{corollary}\label{cor:terms-in-hqu-hqv}
    Consider non-negative integers $r_1,\cdots,r_n\leq 2q-2$. We have that $z_1^{r_1}\cdots z_n^{r_n}$ is a term in $h_u^{(q)}\cdot h_v^{(q)}$ for some $u,v$ with $u-v\geq\Delta$ if and only if
    \[ \sum_{i=1}^n \theta_q(r_i) \geq \Delta.\]
\end{corollary}

We are now ready to explain the characterization of WLP over a field of characteristic $2$.

\begin{proof}[Proof of Theorem~\ref{thm:WLP-char-2}]
    Using the identity
    \[ s^{(q)}_{(u,v)} = h^{(q)}_u\cdot h^{(q)}_v - h^{(q)}_{u+1}\cdot h^{(q)}_{v-1} \quad\text{ for }u\geq v, \]
    it follows that
    \begin{equation}\label{eq:telescope-qschur} \sum_{j\geq 0} s^{(q)}_{(u+j,v-j)} = h^{(q)}_u\cdot h^{(q)}_v.
    \end{equation}
Recall that failure of WLP is characterized by the support sets \eqref{eq:def-supp-d-e}, which satisfy \eqref{eq:incl-supp-de}. In particular, the condition
\[ \ul{a} \in \op{supp}(d,e) = \bigcup_{f\geq 0} \op{supp}(d-f,e+f)\quad\text{ for }e\geq d-1\]
is equivalent to the fact that $z_1^{a_1}\cdots z_n^{a_n}$ appears as a term in
\[ \sum_{f\geq 0} h^1(D^{d-f}\mc{R}(e+f)) = \sum_{\substack{q=2^r\geq 2 \\ m,j\geq 0}} F^{2q}(\mc{N}_m)\cdot h^{(q)}_{e-(2m-2j-1)q}\cdot h^{(q)}_{d-(2m+2j+1)q}\]
where the last equality follows by combining Corollary~\ref{cor:coh-Ddr-using-Nim} with \eqref{eq:telescope-qschur}. Writing $u=e-(2m-2j-1)q$, $v=d-(2m+2j+1)q$, we note that 
\[ u-v = e - d + (4j+2)q \geq 2q-1\text{ if }e\geq d-1.\]
In particular, we conclude that $M_{\ul{a}}$ fails WLP if and only if
\[ z_1^{a_1}\cdots z_n^{a_n}\text{ appears as a term in }F^{2q}(\mc{N}_m)\cdot h_u^{(q)}\cdot h_v^{(q)}\text{ for some }q=2^k\geq 2,\ u,v\geq 0,\text{ with }u-v\geq 2q-1.\]

Consider now a tuple $\ul{a}$ for which the condition above holds, and define $b_i,r_i$ such that
\begin{equation}\label{eq:def-bi-ri} a_i = 2qb_i + r_i,\ 0\leq r_i <2q\quad\text{ for }i=1,\cdots,n.
\end{equation}
It follows that $z_1^{b_1}\cdots z_n^{b_n}$ is a term in $\mc{N}_m$, and $z_1^{r_1}\cdots z_n^{r_n}$ is a term in $h_u^{(q)}\cdot h_v^{(q)}$. Since the Nim sum of $b_1,\cdots,b_n$ is zero, and the powers of $2$ that appear in the $2$-adic expansion of each $r_i$ are bounded by $q$, we get
\[ a_1\oplus \cdots \oplus a_n = r_1 \oplus \cdots \oplus r_n \leq 1 + 2 + 4  + \cdots + q = 2q-1 < 2q.\]
Moreover, applying Corollary~\ref{cor:terms-in-hqu-hqv} with $\Delta=2q-1$, we get
\[ \sum_{i=1}^n \theta_q(a_i) = \sum_{i=1}^n \theta_q(r_i) \geq 2q-1.\]
It follows that if $M_{\ul{a}}$ does not satisfy WLP, then condition \eqref{eq:WLP-cond-char2} must fail for some value of $q=2^k$.

Conversely, suppose that \eqref{eq:WLP-cond-char2} fails for some $q=2^k$, and define $b_i,r_i$ via \eqref{eq:def-bi-ri}. We have as before that $r_1\oplus\cdots\oplus r_n < 2q$, hence
\[ 2q > a_1 \oplus \cdots \oplus a_n = 2q(b_1\oplus\cdots\oplus b_n) + (r_1\oplus\cdots\oplus r_n),\]
which forces $b_1\oplus\cdots\oplus b_n=0$. The failure of \eqref{eq:WLP-cond-char2} translates into
\[ 2q - 1 \leq \sum_{i=1}^n \theta_q(a_i) = \sum_{i=1}^n \theta_q(r_i), \]
which implies $\theta_q(r_i)\neq -\infty$, and hence $r_i\leq 2q-2$ for all $i$. We can then apply Corollary~\ref{cor:terms-in-hqu-hqv} to conclude that $z_1^{r_1}\cdots z_n^{r_n}$ appears as a term in $h^{(q)}_u\cdot h^{(q)}_v$ for some $u-v\geq 2q-1$, and letting $m=b_1+\cdots + b_n$, we obtain that $z_1^{a_1}\cdots z_n^{a_n}$ appears as a term in $F^{2q}(\mc{N}_m)\cdot h_u^{(q)}\cdot h_v^{(q)}$, hence $WLP$ fails for $M_{\ul{a}}$, concluding our proof.
\end{proof}

\section*{Acknowledgements}
Experiments with Macaulay2 \cite{GS} have provided many valuable insights. The authors would like to thank Zhao Gao, Evan O'Dorney, Jenny Kenkel, Alessio Sammartano, Anurag Singh, Keller VandeBogert for helpful discussions regarding various aspects of this project. Marangone gratefully acknowledges that this research was supported in part by the Pacific Institute for the Mathematical Sciences. Raicu and Reed acknowledge the support of the National Science Foundation Grant DMS-2302341. Part of the material in this paper is based upon work supported by the National Science Foundation under Grant No. DMS-1928930 and by the Alfred P. Sloan Foundation under grant G-2021-16778, while Raicu and Reed were in residence at the Simons Laufer Mathematical Sciences Institute (formerly MSRI) in Berkeley, California, during the Spring 2024 semester. Preliminary work on this project was done during the \emph{Pragmatic Research school in Algebraic Geometry and Commutative Algebra} at the University of Catania in June 2023, when all the authors were in residence -- we thank Francesco Russo for his hospitality, and for providing us with the opportunity to participate.


\begin{bibdiv}
    \begin{biblist}

\bib{andersen-b2}{article}{
   author={Andersen, Henning Haahr},
   title={On the structure of the cohomology of line bundles on $G/B$},
   journal={J. Algebra},
   volume={71},
   date={1981},
   number={1},
   pages={245--258},
}

\bib{andersen-survey}{article}{
   author={Andersen, Henning Haahr},
   title={Representation theory via cohomology of line bundles},
   journal={Transform. Groups},
   volume={28},
   date={2023},
   number={3},
   pages={1033--1058},
   issn={1083-4362},
}

\bib{and-kan}{article}{
   author={Andersen, Henning Haahr},
   author={Kaneda, Masaharu},
   title={Cohomology of line bundles on the flag variety for type $G_2$},
   journal={J. Pure Appl. Algebra},
   volume={216},
   date={2012},
   number={7},
   pages={1566--1579},
}

\bib{cook}{article}{
   author={Cook, David, II},
   title={The Lefschetz properties of monomial complete intersections in
   positive characteristic},
   journal={J. Algebra},
   volume={369},
   date={2012},
   pages={42--58},
}

\bib{donkin}{article}{
   author={Donkin, Stephen},
   title={The cohomology of line bundles on the three-dimensional flag
   variety},
   journal={J. Algebra},
   volume={307},
   date={2007},
   number={2},
   pages={570--613},
}

\bib{gao-raicu}{article}{
   author={Gao, Zhao},
   author={Raicu, Claudiu},
   title={Cohomology of line bundles on the incidence correspondence},
   journal={Trans. Amer. Math. Soc. Ser. B},
   volume={11},
   date={2024},
   pages={64--97},
}

\bib{GRV}{article}{
    author={Gao, Zhao},
    author={Raicu, Claudiu},
    author={VandeBogert, Keller},
    title={Some questions arising from the study of cohomology on flag varieties},
    conference={
      title={Open Problems in Algebraic Combinatorics},
    },
    book={
      series={Proceedings of Symposia in Pure Mathematics},
      volume={110},
      publisher={American Mathematical Society},
    },
    pages={333--348},
    date={2024}
}

\bib{GS}{article}{
    author = {Grayson, Daniel R.},
    author = {Stillman, Michael E.},
    title = {Macaulay 2, a software system for research in algebraic geometry}, 
    journal = {Available at \url{http://www.math.uiuc.edu/Macaulay2/}}
}

\bib{griffith}{article}{
   author={Griffith, Walter Lawrence, Jr.},
   title={Cohomology of flag varieties in characteristic $p$},
   journal={Illinois J. Math.},
   volume={24},
   date={1980},
   number={3},
   pages={452--461},
}

\bib{han-monsky}{article}{
   author={Han, C.},
   author={Monsky, P.},
   title={Some surprising Hilbert-Kunz functions},
   journal={Math. Z.},
   volume={214},
   date={1993},
   number={1},
   pages={119--135},
}

\bib{HMNW}{article}{
   author={Harima, Tadahito},
   author={Migliore, Juan C.},
   author={Nagel, Uwe},
   author={Watanabe, Junzo},
   title={The weak and strong Lefschetz properties for Artinian
   $K$-algebras},
   journal={J. Algebra},
   volume={262},
   date={2003},
   number={1},
   pages={99--126},
}

\bib{humphreys-G2}{article}{
   author={Humphreys, J. E.},
   title={Cohomology of line bundles on $G/B$ for the exceptional group
   $G_2$},
   booktitle={Proceedings of the Northwestern conference on cohomology of
   groups (Evanston, Ill., 1985)},
   journal={J. Pure Appl. Algebra},
   volume={44},
   date={1987},
   number={1-3},
   pages={227--239},
}

\bib{kumar}{article}{
   author={Kumar, Shrawan},
   title={Equivariant analogue of Grothendieck's theorem for vector bundles
   on $\PP^1$},
   conference={
      title={A tribute to C. S. Seshadri},
      address={Chennai},
      date={2002},
   },
   book={
      series={Trends Math.},
      publisher={Birkh\"{a}user, Basel},
   },
   isbn={3-7643-0444-2},
   date={2003},
   pages={500--501},
   review={\MR{2017599}},
}

\bib{liu-flag1}{article}{
   author={Liu, Linyuan},
   title={On the cohomology of line bundles over certain flag schemes},
   journal={J. Combin. Theory Ser. A},
   volume={182},
   date={2021},
   pages={Paper No. 105448, 25},
}

\bib{liu-polo}{article}{
   author={Liu, Linyuan},
   author={Polo, Patrick},
   title={On the cohomology of line bundles over certain flag schemes II},
   journal={J. Combin. Theory Ser. A},
   volume={178},
   date={2021},
   pages={Paper No. 105352, 11},
}

\bib{liu}{article}{
   author={Liu, Linyuan},
   title={Cohomologie des fibr\'es en droites sur ${\rm SL}_3/B$ en
   caract\'eristique positive: deux filtrations et cons\'equences},
   language={French, with English and French summaries},
   journal={J. Eur. Math. Soc. (JEMS)},
   volume={26},
   date={2024},
   number={4},
   pages={1365--1422},
}

\bib{lund-nick}{article}{
   author={Lundqvist, Samuel},
   author={Nicklasson, Lisa},
   title={On the structure of monomial complete intersections in positive
   characteristic},
   journal={J. Algebra},
   volume={521},
   date={2019},
   pages={213--234},
}

\bib{maak}{article}{
   author={Maakestad, Helge},
   title={Modules of principal parts on the projective line},
   journal={Ark. Mat.},
   volume={42},
   date={2004},
   number={2},
   pages={307--324},
}

\bib{MMN}{article}{
   author={Migliore, Juan C.},
   author={Mir\'o-Roig, Rosa M.},
   author={Nagel, Uwe},
   title={Monomial ideals, almost complete intersections and the weak
   Lefschetz property},
   journal={Trans. Amer. Math. Soc.},
   volume={363},
   date={2011},
   number={1},
   pages={229--257},
}

\bib{MZ}{article}{
   author={Migliore, Juan},
   author={Zanello, Fabrizio},
   title={The Hilbert functions which force the weak Lefschetz property},
   journal={J. Pure Appl. Algebra},
   volume={210},
   date={2007},
   number={2},
   pages={465--471},
}

\bib{nick}{article}{
   author={Nicklasson, Lisa},
   title={The strong Lefschetz property of monomial complete intersections
   in two variables},
   journal={Collect. Math.},
   volume={69},
   date={2018},
   number={3},
   pages={359--375},
}

\bib{odorney}{article}{
    author={O'Dorney, Evan},
    title={Even-carry polynomials and cohomology of line bundles on the incidence correspondence in positive characteristic},
    journal = {arXiv},
    number = {2404.04166},
    date={2024}
}

\bib{rai-surv}{article}{
   author={Raicu, Claudiu},
   title={Homological invariants of determinantal thickenings},
   journal={Bull. Math. Soc. Sci. Math. Roumanie (N.S.)},
   volume={60(108)},
   date={2017},
   number={4},
   pages={425--446},
}

\bib{rai-vdb}{article}{
   author={Raicu, Claudiu},
   author={VandeBogert, Keller},
   title={Stable sheaf cohomology on flag varieties},
   journal={J. Amer. Math. Soc.},
   volume={39},
   date={2026},
   number={4},
   pages={985--1026},
}

\bib{singh}{article}{
   author={Singh, Anurag K.},
   title={$p$-torsion elements in local cohomology modules},
   journal={Math. Res. Lett.},
   volume={7},
   date={2000},
   number={2-3},
   pages={165--176},
}

\bib{stanley}{article}{
   author={Stanley, Richard P.},
   title={Weyl groups, the hard Lefschetz theorem, and the Sperner property},
   journal={SIAM J. Algebraic Discrete Methods},
   volume={1},
   date={1980},
   number={2},
   pages={168--184},
}

\bib{sun}{article}{
   author={Sun, Xiaotao},
   title={Direct images of bundles under Frobenius morphism},
   journal={Invent. Math.},
   volume={173},
   date={2008},
   number={2},
   pages={427--447},
}

\bib{watanabe}{article}{
   author={Watanabe, Junzo},
   title={The Dilworth number of Artinian rings and finite posets with rank
   function},
   conference={
      title={Commutative algebra and combinatorics},
      address={Kyoto},
      date={1985},
   },
   book={
      series={Adv. Stud. Pure Math.},
      volume={11},
      publisher={North-Holland, Amsterdam},
   },
   date={1987},
   pages={303--312},
}
    \end{biblist}
\end{bibdiv}
\end{document}